\documentclass[10pt]{amsart}

\usepackage{amsthm}
\usepackage{amsmath, amssymb, graphicx}
\usepackage{color}
\usepackage{pst-all}

\psset{unit=1mm}

\theoremstyle{plain}
\newtheorem{prop}{Proposition}[section]

\newtheorem{coro}[prop]{Corollary}
\newtheorem{lemm}[prop]{Lemma}

\theoremstyle{definition}
\newtheorem{conv}[prop]{Convention}
\newtheorem{defi}[prop]{Definition}

\newtheorem{exam}[prop]{Example}
\newtheorem{rema}[prop]{Remark}

\numberwithin{equation}{section}

\allowdisplaybreaks

\psset{unit=1mm}
\psset{fillcolor=white}
\psset{dotsep=1.5pt}
\psset{dash=1.5pt 1.5pt}
\psset{linewidth=0.8pt}
\psset{arrowsize=3.5pt}
\psset{doublesep=1pt}
\psset{coilwidth=1.2mm}
\psset{coilheight=2}
\psset{coilaspect=20}
\psset{coilarm=2}
\newpsstyle{eps}{linewidth=0.5pt, doubleline=true}
\newpsstyle{double}{linewidth=0.5pt, doubleline=true}
\newpsstyle{etc}{linestyle=dotted}
\newpsstyle{exist}{linestyle=dashed}
\newpsstyle{thick}{linewidth=1.5pt}
\newpsstyle{thickexist}{linewidth=1.5pt, linestyle=dashed}
\newpsstyle{thin}{linewidth=0.5pt}
\newpsstyle{thinexist}{linewidth=0.5pt, linestyle=dashed}

\def\BPoint(#1,#2,#3){\cnode[style=thin,fillcolor=black,fillstyle=solid](#1,#2){0.5}{#3}}
\def\GPoint(#1,#2,#3){\cnode[style=thin,fillcolor=lightgray,fillstyle=solid](#1,#2){0.5}{#3}}
\def\WPoint(#1,#2,#3){\cnode[style=thin,fillcolor=white,fillstyle=solid](#1,#2){0.5}{#3}}
\def\AArrow(#1,#2){\ncline[nodesep=0.5mm, linewidth=0.8pt, border=1.2pt]{->}{#1}{#2}}
\def\BArrow(#1,#2){\ncline[linecolor=blue, linewidth=1.2pt, linestyle=dotted, nodesep=0.5mm, border=1.2pt]{->}{#1}{#2}}
\def\CArrow(#1,#2){\ncline[linecolor=red, nodesep=0.5mm, linewidth=0.8pt, border=1.2pt]{->}{#1}{#2}}
\def\DArrow(#1,#2){\ncline[linecolor=red,style=thickexist,nodesep=0.5mm,border=2pt]{->}{#1}{#2}}

\def\Arrow(#1,#2){\ncline[nodesep=1mm,border=2pt]{->}{#1}{#2}}
\def\ArrowE(#1,#2){\ncline[nodesep=1mm,border=2pt,style=exist]{->}{#1}{#2}}
\def\ETC(#1,#2){\ncline[nodesep=2mm,border=2pt,style=etc]{#1}{#2}}

\def\Point(#1,#2,#3){\pnode(#1,#2){#3}}
\def\PointN(#1,#2,#3,#4){\pnode(#1,#2){#3}{#4}}

\renewcommand\aa{a}

\newcommand\bb{b}
\newcommand\BB{B}
\newcommand\BP[1]{B^{\scriptscriptstyle+}_{#1}}
\newcommand\card[1]{\mathtt{\#}{#1}}
\newcommand\cc{c}

\newcommand\comp{\mathbin{\scriptscriptstyle\circ}}
\newcommand\dd{d}

\newcommand\Deltaz{\Delta\refer}
\newcommand\Div{\mathrm{Div}}
\newcommand\dive{\mathbin{\preccurlyeq}}

\newcommand\Double{\Psi}
\newcommand\ee{e}
\newcommand\equivz{\mathrel{\equiv\refer}}
\newcommand\ew{\varepsilon}

\newcommand\ff{f}

\renewcommand\ge{\geqslant\nobreak}
\renewcommand\gg{g}
\newcommand\GG{G}
\newcommand\gga{a}
\newcommand\GGb{\overline\GG}

\newcommand\GL{\mathrm{GL}}
\newcommand\hh{h}
\newcommand\hha{b}
\newcommand\HS[1]{\hspace{#1ex}}
\newcommand\ID{\mathrm{id}}
\newcommand\ii{i}
\newcommand\II{I}
\newcommand\inv{^{-1}}

\newcounter{ITEM}
\newcommand\ITEM[1]{\setcounter{ITEM}{#1}\leavevmode\hbox{\rm(\roman{ITEM})}}
\newcommand\jj{j}
\newcommand\JJ{J}
\newcommand\kk{k}
\newcommand\LCM[2]{\mathrm{lcm}(#1,#2)}
\renewcommand\le{\leqslant\nobreak}

\newcommand\leSb{\mathrel{\le_{\HS{-0.7}\SSb}}}

\newcommand\LGG[2]{\Vert#2\Vert_{\!_{#1}}}

\newcommand\linrep{\Theta}
\newcommand\linrepz{\Theta_{\HS{-0.2}\scriptscriptstyle\bullet}}
\newcommand\mm{m}
\newcommand\MM{M}
\newcommand\MMb{M_{\HS{-0.2}\scriptscriptstyle\bullet}}

\newcommand\Nb[2]{\mathtt{\#}_{#1}(#2)}
\newcommand\nn{n}
\newcommand\NN{N}
\newcommand\NNNN{\mathbb{N}}
\newcommand\nub{\overline\nu}
\newcommand\op{*}
\newcommand\OP{\mathbin{\scriptstyle\bullet}}
\newcommand\opl{\HS{0.2}\rceil\HS{0.1}}
\newcommand\opL{\mathbin{\tilde\op}}
\newcommand\opp{^{\scriptscriptstyle\mathrm{opp}}}
\newcommand\opr{\HS{0.1}\lceil\HS{0.2}}

\newcommand\pdots{\hspace{0.2ex}{\cdot}{\cdot}{\cdot}\hspace{0.2ex}}
\newcommand\piz{\pi\refer}
\newcommand\pp{p}
\newcommand\PP{P}
\newcommand\Pol{\Omega}
\newcommand\Polt{\widetilde\Pol}
\newcommand\Pres[2]{(#1, #2)}
\newcommand\PRES[2]{\langle#1\,\vert\, #2\rangle}
\newcommand\PRESp[2]{\langle#1\,\vert\, #2\rangle^{\scriptscriptstyle\!+}\!}\newcommand\Pw{\mathfrak{P}}
\newcommand\Pwfin{\Pw_{\HS{-0.3}f\HS{-0.2}i\HS{-0.2}n\HS{-0.3}}}
\newcommand\qq{q}
\newcommand\QQQQ{\mathbb{Q}}
\def\resp{\mbox{\it resp}.\ }
\newcommand\RC{\theta}
\newcommand\refer{_{\scriptscriptstyle\bullet}}
\newcommand\rr{r}
\newcommand\RR{R}
\newcommand\RRRR{\mathbb{R}}
\newcommand\sep{\HS{0.05}\vert\HS{0.05}}
\newcommand\Seq[2]{#1^{[#2]}}
\newcommand\sig[1]{\sigma_{\hspace{-0.2ex}#1}^{\null}}

\newcommand\sigg[2]{\sigma_{#1}^{#2}}

\newcommand\sigmaz{\sigma\refer}
\newcommand\sol{\rho}

\newcommand\SPol{\Pi}
\newcommand\SPolt{\widetilde\SPol}
\renewcommand\ss{s}
\renewcommand\SS{S}
\newcommand\SSb{\overline\SS}
\newcommand\sst{\widetilde\ss}
\newcommand\Sym{\mathfrak{S}}
\renewcommand\tt{t}
\newcommand\tta{\mathtt{a}}

\newcommand\ttb{\mathtt{b}}

\newcommand\ttc{\mathtt{c}}

\newcommand\tts{\mathtt{s}}
\newcommand\under{\backslash}
\newcommand\uu{u}
\newcommand\vv{v}
\newcommand\VV{V}
\def\VR(#1,#2){\vrule width0pt height#1mm depth#2mm}
\newcommand\wdots{, ...\hspace{0.2ex},}

\newcommand\WW{W}
\newcommand\xx{x}
\newcommand\yy{y}
\newcommand\Zd{\ZZZZ{/}\dd\ZZZZ}
\newcommand\zz{z}
\newcommand\ZZZZ{\mathbb{Z}}

\begin{document}

\title[Yang--Baxter equation, RC-calculus, and Garside germs]{Set-theoretic solutions of the Yang--Baxter equation, RC-calculus, and Garside germs}

\author{Patrick DEHORNOY}
\address{Laboratoire de Math\'ematiques Nicolas Oresme, UMR 6139 CNRS, Universit\'e de Caen BP 5186, 14032 Caen Cedex, France}
\curraddr{\sc Laboratoire Preuves, Programmes, Syst\`emes, UMR 7126 CNRS, Universit\'e Paris-Diderot Case 7014, 75205 Paris Cedex 13, France}
\email{dehornoy@math.unicaen.fr}
\urladdr{//www.math.unicaen.fr/\!\hbox{$\sim$}dehornoy}



\maketitle

\begin{abstract}
Building on a result by W.\,Rump, we show how to exploit the right-cyclic law $(\xx \yy) (\xx \zz) = (\yy \xx) (\yy \zz)$ in order to investigate the structure groups and monoids attached with (involutive nondegenerate) set-theoretic solutions of the Yang--Baxter equation. We develop a sort of right-cyclic calculus, and use it to obtain short proofs for the existence both of the Garside structure and of the $I$-structure of such groups. We describe finite quotients that play for the considered groups the role that Coxeter groups play for Artin--Tits groups. 
\end{abstract}

The Yang--Baxter equation (YBE) is a fundamental equation occurring in integrable models in statistical mechanics and quantum field theory~\cite{Jim}. Among its many solutions, some simple ones called set-theoretic turn out to be connected with several interesting algebraic structures. In particular, a group and a monoid are attached with every set-theoretic solution of~YBE~\cite{Eti}, and the family of all groups and monoids arising in this way is known to have rich properties: as shown by T.\,Gateva--Ivanova and M.\,Van den Bergh in~\cite{Gav} and by E.\,Jespers and J.\,Okni\'nski in~\cite{JeO}, they admit an $I$-structure, meaning that their Cayley graph is isometric to that of a free Abelian group, and, as shown by F.\,Chouraqui in~\cite{Cho}, they admit a Garside structure, (roughly) meaning that they are groups of fractions of monoids in which divisibility relations are lattice orders. 

It was shown by W.\,Rump in~\cite{RumYB} that (involutive nondegenerate) set-theoretic solutions of YBE are in one-to-one correspondence with algebraic structures consisting of a set equipped with a binary operation~$\op$ that obeys the right-cyclic law $(\xx\yy)(\xx\zz) = (\yy\xx)(\yy\zz)$ and has bijective left-translations. In this paper, we merge the ideas stemming from the right-cyclic law (RC-law) and those coming from Garside theory to give easy alternative proofs of earlier results and derive new results. The key technical point is the connection between the RC-law and the least common right-multiple operation. A nice point is that one never needs to restrict to squarefree solutions of YBE, that is, those satisfying $\sol(\ss, \ss) = (\ss, \ss)$. 

The main benefit of the current approach is to provide a simple and complete solution to the problem of finding a \emph{Garside germ} for every group associated with a set-theoretic solution of YBE, namely finding a finite quotient of the group that encodes the whole structure in the way a finite Coxeter group encodes the associated Artin--Tits group. The precise statement (Proposition~\ref{P:Cox}) says that, if $(\SS, \op)$ is an RC-quasigroup with cardinality~$\nn$ and class~$\dd$ (a certain numerical parameter attached with every finite RC-quasigroup), then starting from the canonical presentation of the associated group and adding the RC-torsion relations~$\ss^{[\dd]} =\nobreak 1$ with~$\ss$ in~$\SS$ (where $\ss^{[\dd]}$ is a sort of twisted $\dd$th power) provides a finite group~$\GGb$ of order~$\dd^\nn$ from which the Garside structure of~$\GG$ can be retrieved. Partial results corresponding to class~$2$, namely RC-quasigroups satisfying $(\xx \xx) (\xx \yy) = \yy$, are established by hand in~\cite{ChG}. Our current approach based on RC-calculus and the $I$-structure enables one to address the general case directly. 

The above finite ``Coxeter-like'' group~$\GGb$ does not coincide (in general) with the finite quotient~$G^0_{\!X}$ considered in~\cite{Eti} and called involutive Yang--Baxter group (IYB) in~\cite{CeJeRiYB, CeJeOkYB}: the latter is a (proper) quotient of the group~$\GGb$ (see Remark~\ref{R:IYB}\ITEM2) and, contrary to~$\GGb$, it does not fully encode the situation since many different solutions may be associated with the same IYB group~\cite{CeJeRiYB}. 

No exhaustive description of Coxeter-like groups is known so far, but it is easy to characterize them as those finite groups that admit a ``modular $I$-structure'', namely the counterpart of an $I$-structure where the free Abelian group~$\ZZZZ$ is replaced with a cyclic group~$\Zd$ (Proposition~\ref{P:CharCox}), a special case of the notion of IG-structure considered in~\cite{GoJe}. If $\GGb$ is associated with a cardinal~$\nn$ RC-quasigroup, its Cayley graph is naturally drawn on an $\nn$-torus, and $\GGb$ can be realized as a group of isometries in an $\nn$-dimensional Hermitian space (Corollary~\ref{C:Isometries}). We hope that further properties will be discovered soon. 

The paper is organized as follows. In Section~\ref{S:Several}, we recall the connection between set-theoretic solutions of the Yang--Baxter equation and algebraic systems that obey the RC-law and introduce the derived structure monoids and groups. In Section~\ref{S:RCCalculus}, we establish various consequences of the RC-law which altogether make a sort of right-cyclic calculus. In Section~\ref{S:StrMonoid}, this calculus is used to investigate the divisibility relations of the monoids associated with RC-quasigroups and their Garside structure. Then, in Section~\ref{S:IStructure} (which is mostly independent from Section~\ref{S:StrMonoid}), we use the RC-calculus to similarly investigate the $I$-structure. Finally, in Section~\ref{S:Coxeter}, we merge the results to construct a finite quotient that encodes the whole structure and give several descriptions of this ``Coxeter-like'' group, in particular as a group of isometries of an Hermitian space. 

\section*{Acknowledgments}
The author wishes to thank Tatiana Gateva-Ivanova, Eddy Godelle, Eric Jespers, Gilbert Levitt, Jan Okni\'nski, and, specially, Wolfgang Rump for discussions about the content of this paper. 

\section{Several equivalent frameworks}
\label{S:Several}

In this introductory section, we recall the definition of set-theoretic solutions of the Yang--Baxter equation~\cite{Eti} and their connection with what we shall call \emph{RC-quasigroups}, which are sets equipped with a binary operation obeying the right-cyclic law $(\xx \op \yy) \op (\xx \op \zz) = (\yy \op \xx) \op (\yy \op \zz)$, as established by W.\,Rump in~\cite{RumYB}.

\begin{defi}\cite{Dri, Eti}
\label{D:SetTheor}
A \emph{set-theoretic solution of YBE} (or \emph{braided quadratic set}) 
is a pair~$(\SS, \sol)$ where $\SS$ is a set and $\sol$ is a bijection of~$\SS \times \SS$ into itself that satisfies
\begin{equation}
\label{E:SetTheor}
\sol^{12} \sol^{23} \sol^{12} = \sol^{23} \sol^{12} \sol^{23}.
\end{equation}
where $\sol^{ij}$ is the map of~$\SS^3$ to itself obtained when~$\sol$ acts on the $\ii$th and~$\jj$th entries. 
\end{defi}

If $(\SS, \sol)$ is a set-theoretic solution of YBE and $\VV$ is a vector space based on~$\SS$, then the (unique) linear operator~$\RR$ on~$\VV \otimes \VV$ that extends~$\sol$ is a solution of the (non-parametric, braid form of) the  Yang--Baxter equation $\RR^{12} \RR^{23} \RR^{12} = \RR^{23} \RR^{12} \RR^{23}$, and, conversely, every solution of YBE such that there exists a basis~$\SS$ of the ambient vector space such that $\SS^{\otimes 2}$ is globally preserved is of this type. 

A set-theoretic solution $(\SS, \sol)$ of YBE is called \emph{nondegenerate} if, writing $\sol_1(\ss, \tt)$ and $\sol_2(\ss, \tt)$ for the first and second entries of~$\sol(\ss, \tt)$, the left-translation $\yy\mapsto \sol_1(\ss,\yy)$ is one-to-one for every~$\ss$ in~$\SS$ and the right-translation $\xx\mapsto \sol_2(\xx,\tt)$ is one-to-one for every~$\tt$ in~$\SS$. A solution $(\SS,\sol)$ is called \emph{involutive} if $\sol\comp\sol$ is the identity of~$\SS \times \SS$. 

A map from~$\SS \times \SS$ to itself is a pair of maps from~$\SS \times \SS$ to~$\SS$, hence a pair of binary operations on~$\SS$. Translating into the language of binary operations the constraints that define set-theoretic solutions of YBE is straightforward.

\begin{lemm}
\label{L:SolRack}
Define a \emph{birack} to be an algebraic system~$(\SS, \opl, \opr)$ consisting of a set~$\SS$ equipped with two binary operations~$\opl$ and~$\opr$ that satisfy
\begin{gather}
\label{E:Rack1}
(\aa\opl\bb) \opl ((\aa\opr\bb) \opl \cc) = \aa \opl (\bb\opl\cc),\\
\label{E:Rack2}
(\aa\opl\bb) \opr ((\aa\opr\bb) \opl \cc) = (\aa \opr (\bb\opl\cc)) \opl (\bb\opr\cc),\\
\label{E:Rack3}
(\aa\opr\bb) \opr \cc = (\aa \opr (\bb\opl\cc)) \opr (\bb\opr\cc),
\end{gather}
and are such that the left-translations of~$\opl$ and the right-translations of~$\opr$ are one-to-one, and call a birack \emph{involutive} if it satisfies in addition
\begin{equation}
\label{E:Rack4}
(\aa\opl\bb) \opl (\aa\opr\bb) = \aa \text{\quad and \quad} (\aa\opl\bb) \opr (\aa\opr\bb) = \bb.
\end{equation}

\ITEM1 If $(\SS, \sol)$ is a nondegenerate set-theoretic solution of YBE, then defining $\aa \opl \bb = \sol_1(\aa, \bb)$ and $\aa \opr \bb = \sol_2(\aa, \bb)$ yields a birack~$(\SS, \opl, \opr)$. If $(\SS, \sol)$ is involutive, then the birack $(\SS, \opl, \opr)$ is involutive. 

\ITEM2 Conversely, if $(\SS, \opl, \opr)$ is a birack, then defining $\sol(\aa, \bb) = (\aa \opl \bb, \aa \opr \bb)$ yields a nondegenerate set-theoretic solution $(\SS, \sol)$ of YBE. If the birack $(\SS, \opl, \opr)$ is involutive, then $(\SS, \sol)$ is involutive.
\end{lemm}

Lemma~\ref{L:SolRack} appears as Remark~1.6 in~\cite{GaMa}, using the notation $(^\aa\bb, \aa^\bb)$ for $(\aa\opl\bb, \aa\opr\bb)$. Biracks appeared in low-dimensional topology as a natural algebraic counterpart of Reidemeister move~III~\cite{FJK}. If $\aa \opr \bb = \aa$ always holds, \eqref{E:Rack1}--\eqref{E:Rack3} reduce to the left-selfdistributivity law $(\aa\opl\bb) \opl (\aa \opl \cc) = \aa \opl (\bb\opl\cc)$, corresponding, when left-translations are bijective, to $(\SS, \opl)$ being a \emph{rack}~\cite{FeR}. Note that a birack obtained from a rack is involutive only if $\ss \opl \tt = \tt$ holds for all~$\ss, \tt$ (trivial birack).

Thus, investigating involutive nondegenerate set-theoretic solutions of YBE and involutive biracks are equivalent tasks. We now make a second step following~\cite{RumYB}. If $\opl$ is a binary operation on~$\SS$ and its left-translations are one-to-one, defining $\aa \op \bb$ to be the unique $\cc$ satisfying $\aa \opl \cc = \bb$ provides a well-defined binary operation on~$\SS$, which can be viewed as a left-inverse of~$\opl$. The seminal observation of~\cite{RumYB} is that, if $(\SS, \opl, \opr)$ is a birack, then the left-inverse~$\op$ of the operation~$\opl$ obeys a simple algebraic law and the whole structure can be recovered from that operation~$\op$.

\begin{defi}[Rump \cite{RumYB}]
\label{D:RC}
A \emph{right-cyclic system}, or \emph{RC-system}, is a pair $(\SS, \op)$ where $\op$ is a binary operation on the set~$\SS$ that obeys the \emph{right-cyclic law}~$RC$ 
\begin{equation}
\label{E:RC}
(\xx \op \yy) \op (\xx \op \zz) = (\yy \op \xx) \op (\yy \op \zz).
\end{equation}
An \emph{RC-quasigroup} is an RC-system whose left-translations are one-to-one, that is, for every~$\ss$ in~$\SS$, the map $\tt \mapsto \ss \op \tt$ is one-to-one. An RC-system is called \emph{bijective} if the map $(\ss, \tt) \mapsto (\ss\op\tt, \tt \op\ss)$ is a bijection of $\SS\times\SS$ to itself.
\end{defi}

In~\cite{RumYB}, RC-systems are called ``cycloids'' and RC-quasigroups are called ``cycle sets''; the current terminology may seem convenient in view of subsequent variants (and the widely used convention that ``quasigroup'' refers to bijective translations).

\begin{exam}
\label{X:RC}
Every operation $\ss\op\tt = \ff(\tt)$ with $\ff$ a permutation of~$\SS$ provides a (semi-trivial) bijective RC-quasigroup. Another example (important in Section~\ref{S:StrMonoid} below) is the right-complement operation in a monoid: if $\MM$ is a left-cancellative monoid in which any two elements admit a unique least common right-multiples, then the operation~$\under$ such that $\ff (\ff \under \gg)$ is the least common right-multiple of~$\ff$ and~$\gg$ obeys~\eqref{E:RC}, as easily follows from the commutativity and associativity of the right-lcm. So $(\MM, \under)$ is an RC-system (but, in general, not an RC-quasigroup).
\end{exam}

The following result, which is essentially~\cite[Prop.~1]{RumYB} shows that the context of a bijective RC-quasigroup is entirely equivalent to that of an involutive nondegenerate set-theoretic solution of the YBE.

\begin{prop}
\label{P:SolRC}
\ITEM1 Assume that $(\SS, \sol)$ is an involutive nondegenerate set-theoretic solution of~YBE. For $\ss, \tt$ in~$\SS$, define $\ss \op \tt$ to be the unique $\rr$ satisfying $\sol_1(\ss, \rr) = \tt$. Then $(\SS, \op)$ is a bijective RC-quasigroup. 

\ITEM2 Conversely, assume that $(\SS, \op)$ is a bijective RC-quasigroup. For~$\aa, \bb$ in~$\SS$, define $\sol(\aa, \bb)$ to be the unique pair~$(\aa', \bb')$ satisfying $\aa \op \aa' = \bb$ and $\aa' \op \aa = \bb'$. Then $(\SS, \sol)$ is an involutive nondegenerate set-theoretic solution of~YBE. 
\end{prop}

As the result essentially appears in~\cite{RumYB}, we shall not go into the details of the proof, whose principle is clear: according to Lemma~\ref{L:SolRack}, the point is to go from the birack laws to the RC-law and \textit{vice versa}, and it consists in repeatedly using the fact that, for all~$\xx, \yy, \zz$ in the considered set~$\SS$, the relation $\yy = \xx \opl \zz$ is equivalent to $\xx\op\yy = \zz$ and it implies $\yy \op \xx = \xx \opr \zz$ (for~\ITEM1) and, symmetrically, that $\zz = \xx \op \yy$ is equivalent to $\xx\opl\zz = \yy$ and it implies $\xx\opr\zz = \yy \op \xx$ (for~\ITEM2). The argument then amounts to completing the cube displayed in Figure~\ref{F:Cube}, where a square diagram 
\VR(0,6)
\begin{picture}(18,8)(-3,3)
\pcline{->}(1,8)(9,8)
\taput{$\aa'$}
\pcline{->}(1,0)(9,0)
\tbput{$\bb$}
\pcline{->}(0,7)(0,1)
\tlput{$\aa$}
\pcline{->}(10,7)(10,1)
\trput{$\bb'$}
\end{picture}
means that we have $\aa' = \aa\opl\bb$ and $\bb' = \aa\opr\bb$, that is, equivalently, $\bb = \aa \op \aa'$ and $\bb' = \aa' \op \aa$. As observed by W.\,Rump, this picture illustrates the nature of~\eqref{E:RC} as a (discrete form of) an integrability condition.

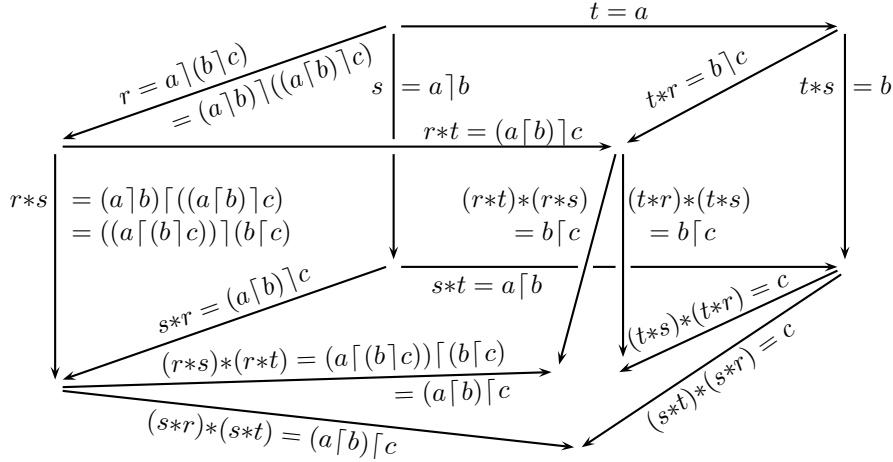
\begin{figure}[htb]
\begin{picture}(105,58)(0,-1)
\pcline{->}(46,56)(104,56)
\taput{$\tt = \aa$}

\pcline{->}(44,56)(1,41)
\put(8,45.5){\rotatebox{20}{\hbox{$\rr = \aa \opl(\bb\opl\cc)$}}}
\put(15,42){\rotatebox{20}{\hbox{$ = (\aa\opl\bb)\opl((\aa\opr\bb)\opl\cc)$}}}

\pcline{->}(0,39)(0,9)
\put(-6,32){$\rr{\op}\ss$}
\put(1,32){ $= (\aa\opl\bb)\opr((\aa\opr\bb)\opl\cc)$}
\put(1,27.5){ $= ((\aa\opr(\bb\opl\cc))\opl(\bb\opr\cc)$}

\pcline(45,55)(45,41)
\pcline{->}(45,39)(45,25)
\put(42,47){$\ss$}
\put(46,47){$ = \aa\opl\bb$}

\pcline{->}(1,40)(74,40)
\put(49,41){$\rr{\op}\tt = (\aa\opr\bb)\opl\cc$}

\pcline{->}(44,24)(1,9)
\put(13,15){\rotatebox{20}{\hbox{$\ss{\op}\rr = (\aa\opr\bb)\opl\cc$}}}

\pcline{->}(1,7.5)(69,0)
\put(12,2.5){\rotatebox{-6}{\hbox{$(\ss{\op}\rr){\op}(\ss{\op}\tt) = (\aa\opr\bb)\opr\cc$}}}

\pcline{->}(1,8)(66,10)
\put(14,10){\rotatebox{2}{\hbox{$(\rr{\op}\ss){\op}(\rr{\op}\tt) = (\aa\opr(\bb\opl\cc))\opr(\bb\opr\cc) $}}}
\put(45,6){\rotatebox{2}{\hbox{$= (\aa\opr\bb)\opr\cc$}}}

\pcline{->}(74.5,39)(67,11)
\put(54,32){$(\rr{\op}\tt){\op}(\rr{\op}\ss)$}
\put(61,27.5){$ = \bb\opr\cc$}

\pcline{->}(75.5,39)(75.5,12)
\put(76,32){$(\tt{\op}\rr){\op}(\tt{\op}\ss)$}
\put(79,27.5){$ = \bb\opr\cc$}

\pcline{->}(105,55)(105,25)
\put(99,47){$\tt{\op}\ss$}
\put(106,47){$ = \bb$}

\pcline{->}(104,55.5)(76,40.5)
\put(78,45){\rotatebox{26}{\hbox{$\tt{\op}\rr = \bb\opl\cc$}}}

\pcline(46,24)(69.5,24)
\pcline(71.5,24)(74.5,24)
\pcline{->}(76.5,24)(104,24)
\put(50,20.5){$\ss{\op}\tt = \aa\opr\bb$}

\pcline{->}(104.5,23)(70,0)
\put(78,2){\rotatebox{35}{\hbox{$(\ss{\op}\tt){\op}(\ss{\op}\rr) = \cc$}}}

\pcline{->}(104,23.5)(75,10)
\put(75.5,13){\rotatebox{22}{\hbox{$(\tt{\op}\ss){\op}(\tt{\op}\rr) = \cc$}}}

\end{picture}
\caption[]{\sf\smaller Proof of Proposition~\ref{P:SolRC}\ITEM1: one successively evaluates the edges of the cube in terms of~$\aa, \bb, \cc$ and the relations~\eqref{E:Rack1}--\eqref{E:Rack3} guarantee that the cube closes. The same diagram can be used for the proof of~\ITEM2 below, except that one starts with a closed cube and evaluates some edges in two different ways to establish~\eqref{E:Rack1}--\eqref{E:Rack3}.}
\label{F:Cube}
\end{figure}

The (small) miracle in Proposition~\ref{P:SolRC} is that it reduces the constraints of a birack, which involve two operations and three laws, to those of an RC-quasigroup, which only involves one operation and one law. Actually, a second, symmetric operation is also present in every bijective RC-quasigroup. 

\begin{defi}
An \emph{RLC-system} is a triple $(\SS, \op, \opL)$ such that $(\SS, \op)$ is an RC-system, $\opL$ is a second binary operation on~$\SS$ that obeys the \emph{left-cyclic law}~$LC$ 
\begin{equation}
\label{E:LC}
(\zz \opL \xx) \opL (\yy \opL \xx) = (\zz \opL \yy) \opL (\xx \opL \yy),
\end{equation}
and both operations are connected by
\begin{equation}
\label{E:InvolRLC}
(\yy \op \xx)\opL(\xx \op \yy) = \xx = (\yy \opL \xx) \op (\xx \opL \yy).
\end{equation}
An \emph{RLC-quasigroup} is an RLC-system~$(\SS, \op, \opL)$ such that the left-translations of~$\op$ and the right-translations of~$\opL$ are one-to-one.
\end{defi}

Then RLC-quasigroups and bijective RC-quasigroups are equivalent structures.

\begin{lemm}
\label{L:SecondOp}
For all binary operations $\op, \opL$ on a set~$\SS$, the following are equivalent:

\ITEM1 The system $(\SS, \op, \opL)$ obeys the involutivity laws~\eqref{E:InvolRLC}.

\ITEM2 The map $\Double : (\ss, \tt) \mapsto (\ss\op\tt, \tt\op\ss)$ is a bijection of $\SS\times\SS$ to itself and $\opL$ is the unique operation on~$\SS$ such that the map $(\ss, \tt) \mapsto (\ss\opL\tt, \tt\opL\ss)$ is the inverse of~$\Double$
\end{lemm}

\begin{proof}
Assume that $(\SS, \op, \opL)$ satisfies~\eqref{E:InvolRLC}. Let $(\ss', \tt')$ belong to~$\SS \times \SS$. Put $\ss = \tt' \opL\ss'$ and $\tt = \ss' \opL\tt'$. By~\eqref{E:InvolRLC}, we have $\ss\op\tt = \ss'$ and $\tt \op\ss = \tt'$, whence $\Double(\ss, \tt) = (\ss', \tt')$. So $\Double$ is surjective. Conversely, assume $\Double(\ss, \tt) = (\ss', \tt')$. Then the left-hand equality in~\eqref{E:InvolRLC} gives $\ss = \tt'\op\ss'$ and $\tt = \ss'\op\tt'$. So $\Double$ is injective. Moreover, the equalities show that the map $(\ss, \tt) \mapsto (\ss\opL\tt, \tt\opL\ss)$ is~$\Double\inv$. So \ITEM1 implies~\ITEM2.

Conversely, if $\Double$ is a bijection from $\SS \times \SS$ to itself, defining $\ss' \opL\tt'$ to be the unique~$\tt$ satisfying $\ss\op\tt = \ss'$ and $\tt\op\ss = \tt'$ for some~$\ss$ guarantees that $(\ss, \tt) \mapsto (\ss\opL\tt, \tt\opL\ss)$ is~$\Double\inv$. Then \eqref{E:InvolRLC} is satisfied by definition, that is, \ITEM2 implies~\ITEM1.
\end{proof}

Summarizing, we see that nondegenerate involutive set-theoretic solutions of the Yang--Baxter equation, involutive biracks, bijective RC-quasi\-groups, and RLC-quasigroups are entirely equivalent frameworks.

\begin{conv}
From now on, we write ``solution of~YBE'' for ``involutive nondegenerate set-theoretic solution of~YBE''.
\end{conv} 

According to~\cite{Eti}, a group and a monoid are associated with every solution of YBE, hence, equivalently, with every (bijective) RC-quasigroup. 

\begin{defi}
\label{D:StrYB}
The \emph{structure group} (\resp \emph{monoid}) associated with a solution~$(\SS, \sol)$ of YBE is the group (\resp \emph{monoid}) defined by the presentation
\begin{equation}
\label{E:StrYB}
\PRES\SS{\{\aa\bb = \aa'\bb' \mid \aa, \bb, \aa', \bb' \in\SS \text{ satisfying } \sol(\aa, \bb) = (\aa', \bb')\}}.
\end{equation}
The \emph{structure group} (\resp \emph{monoid}) associated with an RC-quasigroup~$(\SS, \op)$ is the group (\resp \emph{monoid}) defined by the presentation
\begin{equation}
\label{E:Str}
\PRES\SS{\{\ss (\ss \op \tt) = \tt (\tt \op \ss) \mid \ss \not= \tt \in \SS\}}.
\end{equation}
\end{defi}

Such monoids and groups will be the main subject of investigation in this paper. The first observation is that, as can be expected, the monoids and groups  associated with solutions of YBE and with bijective RC-quasigroups coincide.

\begin{lemm}
\label{L:Str}
If a solution $(\SS, \sol)$ of YBE and a bijective RC-quasigroup $(\SS, \op)$ are connected as in Proposition~\ref{P:SolRC}, then the structure monoids of~$(\SS, \sol)$ and~$(\SS, \op)$ coincide, and so do the corresponding groups.
\end{lemm}

The result directly follows from the connection of Proposition~\ref{P:SolRC}, which shows that the relations of~\eqref{E:StrYB} and~\eqref{E:Str} coincide, up to some repetitions and adding trivial relations $\aa\bb = \aa\bb$ in~\eqref{E:StrYB}. 

Thus investigating structure monoids of solutions of YBE and of bijective RC-quasigroups are equivalent tasks. The goal of this paper is to show the advantages of the second approach. 

\section{RC-calculus}
\label{S:RCCalculus}

In order to exploit the RC-law, we introduce sorts of polynomials involving the RC-operation and establish various algebraic relations that will be heavily used in the sequel. Most verifications are easy, but introducing convenient notation is important to obtain simple formulas and perform computations that, otherwise, would require tedious developments. It should be mentioned that some formulas admit counterparts in the world of braces~\cite{RumBM, RumBr}, which are equivalent to linear RC-systems, defined to be RC-systems equipped with a compatible abelian group operation (by~\cite[Prop.~6]{RumYB}, the group~$\GG$ associated with a bijective RC-quasigroup~$(\SS, \op)$ is the brace~$\ZZZZ^{(\SS)}$ with the Jacobson circle operation).

Everywhere in the sequel, $\opL$, $\op$, and $\cdot$ refer to binary operations.

\begin{defi}
\label{D:Pol}
For $\nn \ge 1$, we inductively define formal expressions $\Pol_\nn(\xx_1 \wdots \xx_\nn)$ and $\Polt_\nn(\xx_\nn \wdots\xx_1)$ by $\Pol_1(\xx_1) = \Polt_1(\xx_1) = \xx_1$ and 
\begin{gather}
\Pol_\nn(\xx_1 \wdots\xx_\nn) = \Pol_{\nn-1}(\xx_1 \wdots\xx_{\nn-1}) \op \Pol_{\nn-1}(\xx_1 \wdots\xx_{\nn-2}, \xx_\nn),\\
\Polt_\nn(\xx_1 \wdots\xx_\nn) = \Polt_{\nn-1}(\xx_1, \xx_3 \wdots\xx_\nn) \opL \Polt_{\nn-1}(\xx_2 \wdots\xx_\nn).
\end{gather}
\end{defi}

The expression $\Pol_\nn(\xx_1 \wdots \xx_\nn)$ is a sort of $\nn$-variable monomial involving~$\op$. We find $\Pol_2(\xx_1, \xx_2) = \xx_1 \op \xx_2$, then $\Pol_3(\xx_1, \xx_2, \xx_3) = (\xx_1 \op \xx_2) \op (\xx_1 \op \xx_3)$, etc. Clearly, $2^{\nn-1}$ variables~$\xx_\ii$ occur in $\Pol_\nn(\xx_1 \wdots\xx_\nn)$, with brackets as in a balanced binary tree. In the language of braces, $\Pol_\nn(\xx_1 \wdots \xx_\nn)$ would correspond to $(\xx_1 + \pdots + \xx_\nn) \op \xx_\nn$ 

If $(\SS, \opL)$ is an algebraic system, $\Pol_\nn(\ss_1 \wdots \ss_\nn)$ is  the evaluation of~$\Pol_\nn(\xx_1 \wdots \xx_\nn)$ when $\xx_\ii$ is given the value~$\ss_\nn$. The next result is an iterated version of the RC-law.

\begin{lemm}
\label{L:Pol}
Assume that $(\SS, \op)$ is an RC-system. Then, for all $\ss_1 \wdots \ss_\nn$ in~$\SS$ and~$\pi$ in~$\Sym_{\nn-1}$, we have $\Pol_\nn(\ss_{\pi(1)} \wdots \ss_{\pi(\nn-1)}, \ss_\nn) = \Pol_\nn(\ss_1 \wdots \ss_\nn)$.
\end{lemm}

It suffices to verify the result when $\pi$ is a transposition of adjacent entries, and then the result follows from an obvious induction on~$\nn \ge 3$, the case $\nn = 3$ precisely corresponding to the RC-law.

Of course, the counterpart of Lemma~\ref{L:Pol} involving~$\Polt_\nn$ is valid when  $\opL$ satisfies the LC-law~\eqref{E:LC}. Further results appear when the monomials~$\Pol_\nn$ are evaluated in an RC-quasigroup, that is, when left-translations are one-to-one.

\begin{lemm}
\label{L:Distinct}
Assume that $(\SS, \op)$ is an RC-quasigroup and $\ss_1 \wdots \ss_\nn$ lie in~$\SS$. 

\ITEM1 The map $\ss \mapsto \Pol_{\nn+1}(\ss_1 \wdots \ss_\nn, \ss)$ is a bijection of~$\SS$ into itself.

\ITEM2 There exist $\rr_1 \wdots \rr_\nn$ in~$\SS$ satisfying $\Pol_\ii(\rr_1 \wdots \rr_\ii) = \ss_\ii$ for $1 \le \ii \le \nn$. 

\ITEM3 Put $\sst_\ii = \Pol_\nn(\ss_1 \wdots \widehat{\ss_\ii}, \wdots \ss_\nn, \ss_\ii)$ for $1 \le \ii \le \nn$. Then, for all~$\ii, \jj$, the relations $\ss_\ii = \ss_\jj$ and $\sst_\ii = \sst_\jj$ are equivalent.
\end{lemm}

\begin{proof}
\ITEM1 Use induction on~$\nn$. For $\nn = 1$, the result directly follows from the assumption. For $\nn \ge 2$, we have $\Pol_{\nn+1}(\ss_1 \wdots \ss_\nn, \ss) = \tt \op \Pol_\nn(\ss_1 \wdots \ss_{\nn-1}, \ss)$  with $\tt = \Pol_\nn(\ss_1 \wdots \ss_{\nn-1})$. By induction hypothesis, $\ss \mapsto \Pol_\nn(\ss_1 \wdots \ss_{\nn-1}, \ss)$ is bijective. Composing with the left-translation by~$\tt$ yields a bijection. 

\ITEM2 Use induction on~$\nn$. For $\nn = 1$, take $\tt_1 = \ss_1$. Assume $\nn \ge 2$. By induction hypothesis, there exist $\rr_1 \wdots \rr_{\nn-1}$ satisfying $\Pol_\ii(\rr_1 \wdots \rr_\ii) = \ss_\ii$ for $1 \le \ii \le \nn-1$. Then, by definition of~$\Pol_\nn$ and owing to $\Pol_{\nn-1}(\rr_1 \wdots \rr_{\nn-1}) = \ss_{\nn-1}$, we have $\Pol_\nn(\rr_1 \wdots \rr_{\nn-1}, \xx) = \ss_{\nn-1} \op \Pol_{\nn-1}(\rr_1 \wdots \rr_{\nn-2}, \xx)$. As the left-translation by~$\ss_{\nn-1}$ is surjective, there exists~$\ss$ satisfying $\ss_{\nn-1} \op \ss = \ss_\nn$. Then, by~\ITEM1, there exists~$\rr_\nn$ satisfying $\Pol_{\nn-1}(\rr_1 \wdots \rr_{\nn-2}, \rr_\nn) = \ss$, whence $\Pol_\nn(\rr_1 \wdots \rr_\nn) = \ss_\nn$.

\ITEM3 Again an induction on~$\nn$. For $\nn = 1$ there is nothing to prove. For $\nn =\nobreak 2$,  we find $\sst_1 = \ss_2 \op \ss_1$ and $\sst_2 = \ss_1 \op \ss_1$, and $\ss_1 = \ss_2$ implies $\sst_1 = \sst_2$. Conversely, assume $\ss_1 \op \ss_2 = \ss_2 \op \ss_1$. Using the assumption twice and the RC-law, we obtain $(\ss_1 \op \ss_2) \op (\ss_2 \op \ss_2) = (\ss_2 \op \ss_1) \op (\ss_2 \op \ss_2) = (\ss_1 \op \ss_2) \op (\ss_1 \op \ss_2) = (\ss_1 \op \ss_2) \op (\ss_2 \op \ss_1)$. As left-translations are injective, we deduce $\ss_2 \op \ss_2 = \ss_2 \op \ss_1$, and $\ss_2 = \ss_1$.  Assume now $\nn \ge 3$. Fix~Ê$\ii, \jj$, write~$\vec\ss$ for $\ss_1 \wdots \widehat{\ss_\ii} \wdots \widehat{\ss_\jj} \wdots \ss_\nn$ and put $\tt_\kk = \Pol_{\nn-1}(\vec\ss, \ss_\kk)$. Then, by~\ITEM1 and by definition, we have  $\sst_\ii = \Pol_\nn(\vec\ss, \ss_\jj, \ss_\ii) = \Pol_{\nn-1}(\vec\ss, \ss_\jj) \op \Pol_{\nn-1}(\vec\ss, \ss_\ii) = \tt_\jj \op \tt_\ii$, and, similarly, $\sst_\jj = \tt_\ii \op \tt_\jj$. If $\ss_\ii = \ss_\jj$ holds, we have $\tt_\ii = \tt_\jj$, whence $\sst_\ii = \sst_\jj$. Conversely, assume $\sst_\ii = \sst_\jj$, that is, $\tt_\jj \op \tt_\ii = \tt_\ii \op \tt_\jj$. By the result for $\nn = 2$, we deduce $\tt_\ii = \tt_\jj$, that is, $\Pol_{\nn-1}(\vec\ss, \ss_\ii) = \Pol_{\nn-1}(\vec\ss, \ss_\jj)$, which is an equality of the form $\rr_1 \op (... \op (\rr_{\nn-2} \op \ss_\ii) ... ) = \rr_1 \op (... \op (\rr_{\nn-2} \op \ss_\jj) ... )$. Repeatedly using the injectivity of left-translations, we deduce $\ss_\ii = \ss_\jj$.
\end{proof}

Further results appear when two operations connected by~\eqref{E:InvolRLC} are involved. In terms of~$\Pol_1$ and~$\Pol_2$, \eqref{E:InvolRLC} says that $\sst_1 = \Pol_2(\ss_1, \ss_2)$ and $\sst_2 = \Pol_2(\ss_2, \ss_1)$ imply $\ss_1 = \Polt_2(\sst_1, \sst_2)$ and $\ss_2 = \Polt_2(\sst_2, \sst_1)$: two elements can be retrieved from their~$\Pol_2$ images using the monomial~$\Polt_2$. Here is an $\nn$-variable version of this result.

\begin{lemm}
\label{L:IterInvol}
Assume that $(\SS, \op, \opL)$ is an RLC-system and $\ss_1 \wdots \ss_\nn$ belong to~$\SS$. For $1 \le \ii \le \nn$, put $\sst_\ii = \Pol_\nn(\ss_1 \wdots \widehat{\ss_\ii}, \wdots \ss_\nn, \ss_\ii)$. Then, for $1 \le \ii \le \nn$, and for every permutation~$\pi$ in~$\Sym_\nn$, we have
\begin{equation}
\label{E:IterInvol}
\Pol_\ii(\ss_{\pi(1)} \wdots \ss_{\pi(\ii)}) = \Polt_{\nn+1-\ii}(\sst_{\pi(\ii)} \wdots \sst_{\pi(\nn)}).
\end{equation}
\end{lemm}

\begin{proof}
For $\nn = 1$, \eqref{E:IterInvol} reduces to $\ss_{\pi(1)} = \ss_{\pi(1)}$. Fix~$\nn \ge 2$ and use induction on~$\ii$ decreasing from~$\nn$ to~$1$. For $\ii = \nn$, \eqref{E:IterInvol} is $\Pol_\nn(\ss_{\pi(1)} \wdots \ss_{\pi(\ii)}) = \Polt_1(\sst_{\pi(\ii)})$. By Lemma~\ref{L:Pol}, the left term is $\Pol_\nn(\ss_1 \wdots \widehat{\ss_\ii} \wdots \ss_{\nn-1}, \ss_{\pi(\ii)})$, hence $\sst_{\pi(\ii)})$, <hence~\eqref{E:IterInvol}. Assume now $\ii < \nn$. Put  $\ss = \Pol_\ii(\ss_{\pi(1)} \wdots \ss_{\pi(\ii)})$, $\ss' = \Pol_\ii(\ss_{\pi(1)} \wdots \ss_{\pi(\ii-1)}, \ss_{\pi(\ii+1)})$, $\tt = \Pol_{\ii+1}(\ss_{\pi(1)} \wdots \ss_{\pi(\ii)}, \ss_{\pi(\ii+1)})$, and $\tt' = \Pol_{\ii+1}(\ss_{\pi(1)} \wdots \ss_{\pi(\ii-1)}, \ss_{\pi(\ii+1)}, \ss_{\pi(\ii)})$. Using the definition of~$\Pol_{\ii+1}$ from~$\Pol_\ii$, we find $\tt = \ss \op \ss'$ and $\tt' = \ss' \op \ss$, whence $\ss = \tt' \opL \tt$ and $\ss' = \tt \opL \tt$ by the involutivity law. Now the induction hypothesis gives
$$\tt = \Polt_{\nn-\ii}(\sst_{\pi(\ii+1)} \wdots \sst_{\pi(\nn)}), \qquad \tt' = \Polt_{\nn-\ii}(\sst_{\pi(\ii)}, \sst_{\pi(\ii+2)} \wdots \sst_{\pi(\nn)}).$$
Using the definition of~$\SPol_{\nn+1-\ii}$ from~$\SPol_{\nn-\ii}$, we find $\ss = \tt' \opL \tt = \Polt_{\nn+1-\ii}(\sst_{\pi(\ii)} \wdots \sst_{\pi(\nn)})$ (and $\ss' = \tt \opL \tt = \Polt_{\nn+1-\ii}(\sst_{\pi(\ii+1)}, \sst_{\pi(\ii)}, \sst_{\pi(\ii+2)} \wdots \sst_{\pi(\nn)})$), which is~\eqref{E:IterInvol}.
\end{proof}

We now introduce terms that involve, in addition to~$\op$ and~$\opL$, a third operation~$\cdot$ that will be evaluated into an associative product.

\begin{defi}
\label{D:SPol}
For $\nn \ge 1$, we introduce the formal expressions
\begin{gather}
\label{E:TBSPol}
\SPol_\nn(\xx_1 \wdots \xx_\nn) = \Pol_1(\xx_1) \cdot \Pol_2(\xx_1, \xx_2) \cdot \pdots \cdot \Pol_\nn(\xx_1 \wdots \xx_\nn)\\
\SPolt_\nn(\xx_1 \wdots \xx_\nn) = \Polt_\nn(\xx_1 \wdots \xx_\nn) \cdot \Polt_{\nn-1}(\xx_2 \wdots \xx_\nn) \cdot \pdots \cdot \Polt_1(\xx_\nn).
\end{gather}
\end{defi}

For $\nn \ge 2$, \eqref{E:TBSPol} implies $\SPol_\nn(\xx_1 \wdots \xx_\nn) = \SPol_{\nn-1}(\xx_1 \wdots \xx_{\nn-1}) \cdot \Pol_\nn(\xx_1 \wdots \xx_\nn)$. We shall consider below a monoid generated by~$\SS$ satisfying $\ss(\ss \op \tt) = \tt(\tt \op \ss)$, that is, $\SPol_2(\ss, \tt) = \SPol_2(\tt, \ss)$. Then we have the following iterated version.

\begin{lemm}
\label{L:SPol}
Assume that $(\SS, \op)$ is an RC-system and $\MM$ is a monoid including~$\SS$ in which $\SPol_2(\ss, \tt) = \SPol_2(\tt, \ss)$ holds for all $\ss, \tt$ in~$\SS$. Then the evaluation of~$\SPol_\nn$ in~$\MM$ is a symmetric function on~$\SS^\nn$.
\end{lemm}

\begin{proof}
We use induction on~$\nn$. For $\nn = 1$, there is nothing to prove. For $\nn = 2$, the result is the equality $\ss_1(\ss_1 \op \ss_2) = \ss_2(\ss_2\op\ss_1)$, which is valid in~$\MM$ by assumption. Assume $\nn \ge 3$. As in Lemma~\ref{L:Pol}, it is sufficient to consider transpositions $(\ii, \ii+\nobreak1)$, that is, to compare $\SPol_\nn(\ss_1 \wdots \ss_\nn)$ and $\SPol_\nn(\ss_1 \wdots \ss_{\ii+1}, \ss_\ii \wdots \ss_\nn)$. By definition, $\SPol_\nn(\ss_1 \wdots \ss_\nn)$ is the product of the values $\Pol_\jj(\ss_1 \wdots \ss_\jj)$ for~$\jj$ increasing from~$1$ to~$\nn$, whereas $\SPol_\nn(\ss_1 \wdots \ss_{\ii+1}, \ss_\ii \wdots \ss_\nn)$ is a similar product of $\Pol_\jj(\ss'_1 \wdots \ss'_\jj)$ with $\ss'_\ii = \ss_{\ii+1}$, $\ss'_{\ii+1} = \ss_\ii$, and $\ss'_\kk = \ss_\kk$ for $\kk \not= \ii, \ii+1$. For $\jj < \ii$, the entries $\ss_\ii$ and~$\ss_{\ii+1}$ do not occur in~$\Pol_\jj(\ss_1 \wdots \ss_\jj)$ and $\Pol_\jj(\ss'_1 \wdots \ss'_\jj)$, which are therefore equal. For $\jj > \ii+1$, the expressions $\Pol_\jj(\ss_1 \wdots \ss_\jj)$ and $\Pol_\jj(\ss'_1 \wdots \ss'_\jj)$ differ by the permutation of two non-final entries, so they are equal by Lemma~\ref{L:Pol}. There remains to compare the central entries $\tt = \Pol_\ii(\ss_1 \wdots \ss_\ii) \cdot \Pol_{\ii+1}(\ss_1 \wdots \ss_{\ii+1})$ and $ \tt' = \Pol_\ii(\ss'_1 \wdots \ss'_\ii) \cdot \Pol_{\ii+1}(\ss'_1 \wdots \ss'_{\ii+1})$. Put $\rr = \Pol_\ii(\ss_1 \wdots \ss_\ii)$ and $\rr' = \Pol_\ii(\ss_1 \wdots \ss_{\ii-1}, \ss_{\ii+1})$. By definition of~$\ss'_\kk$, we have also $\rr = \Pol_\ii(\ss'_1 \wdots \ss'_{\ii-1}, \ss'_{\ii+1})$ and $\rr' = \Pol_\ii(\ss'_1 \wdots \ss'_\ii)$. Then, by definition of~$\Pol_\ii$ and~$\Pol_{\ii+1}$, we have $\tt = \rr(\rr\op\rr')$ and $\tt' = \rr'(\rr' \op\rr)$, whence $\tt = \tt'$ in~$\MM$.
\end{proof}

Lemma~\ref{L:SPol} says in particular that, when we start with $\nn$~elements $\ss_1 \wdots \ss_\nn$ and construct in the Cayley graph of~$\MM$ the $\nn$-cube displayed in Figure~\ref{F:CubePol}, then the cube converges to a unique final vertex and all maximal paths represent~$\SPol(\ss_1 \wdots \ss_\nn)$.

\begin{figure}[htb]
\begin{picture}(90,38)(0,-2)
\pcline{->}(1,15.5)(29.5,30)
\put(7,25){$\Pol_1(\ss_1)$}
\pcline{->}(1,14.5)(29.5,0)
\put(7,3){$\Pol_1(\ss_3)$}
\pcline{->}(1,15)(29.5,15)
\put(12,16.5){$\Pol_1(\ss_2)$}
\pcline{->}(31,30)(59,30)
\taput{$\Pol_2(\ss_1, \ss_2)$}
\pcline{->}(31,15.5)(59,29.5)
\put(25,21){$\Pol_2(\ss_2, \ss_1)$}
\pcline[style=exist]{->}(31,29.5)(59,15.5)
\put(50,21){$\Pol_2(\ss_1, \ss_3)$}
\pcline[style=exist]{->}(31,0.5)(59,14.5)
\put(50,7){$\Pol_2(\ss_3, \ss_1)$}
\pcline{->}(31,14.5)(59,0.5)
\put(25,7){$\Pol_2(\ss_2, \ss_3)$}
\pcline{->}(31,0)(59,0)
\tbput{$\Pol_2(\ss_3, \ss_2)$}
\pcline{->}(60.5,30)(89,15.5)
\put(72,25){$\Pol_3(\ss_1, \ss_2, \ss_3)$}
\pcline[style=exist]{->}(60.5,15)(89,15)
\put(62,16.5){$\Pol_3(\ss_1, \ss_3, \ss_2)$}
\pcline{->}(60.5,0)(89,14.5)
\put(72,3){$\Pol_3(\ss_2, \ss_3, \ss_1)$}
\end{picture}
\caption[]{\sf\smaller The monomials $\Pol_\ii$ occur at the $\ii$th level in an $\nn$-cube built from~$\ss_1 \wdots \ss_\nn$ using~$\op$ to form elementary squares (here $\nn =\nobreak 3$).}
\label{F:CubePol}
\end{figure}
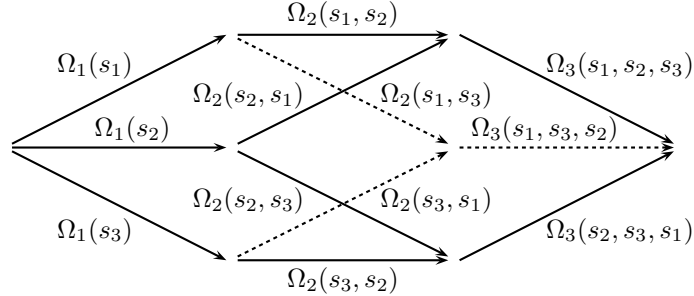

\begin{lemm}
\label{L:IterInvol2}
Assume that $(\SS, \op, \opL)$ is an RLC-system and $\MM$ is a monoid including~$\SS$ in which $\SPol(\ss, \tt) = \SPol_2(\tt, \ss)$ holds for all~$\ss, \tt$ in~$\SS$. Then, for all $\ss_1 \wdots \ss_\nn$ in~$\SS$, the equality $\SPol_\nn(\ss_1 \wdots \ss_\nn) = \SPolt_\nn(\sst_1 \wdots \sst_\nn)$ holds for $\sst_\ii = \Pol_\nn(\ss_1 \wdots \widehat{\ss_\ii}, \wdots \ss_\nn, \ss_\ii)$.
\end{lemm}

\begin{proof}
Using~\eqref{E:IterInvol} and the definitions of~$\SPol_\nn$ and $\SPolt_\nn$, we obtain
\begin{align*}
\SPol_\nn(\ss_1 \wdots \ss_\nn) 
&= \Pol_1(\ss_1) \cdot \Pol_2(\ss_1, \ss_2) \cdot \pdots \cdot \Pol_\nn(\ss_1 \wdots \ss_\nn)\\
&= \Polt_\nn(\sst_1 \wdots \sst_\nn) \cdot \Polt_{\nn-1}(\sst_2 \wdots \sst_\nn) \cdot \pdots \cdot \Polt_1(\sst_\nn) = \SPolt_\nn(\sst_1 \wdots \sst_\nn).
\rlap{\HS{1.2}$\square$}
\end{align*}
\let\qed=\relax\end{proof}
 
\section{The Garside structure}
\label{S:StrMonoid}

It was proved by F.\,Chouraqui~\cite{Cho} that the monoids associated with solutions of YBE (hence with bijective RC-quasigroups) are Garside monoids~\cite{Dgk} and that, conversely, every Garside monoid with a certain type of presentation arises in this way---see also \cite{GatGar}. Here we show how to easily derive such results from the computations of Section~\ref{S:RCCalculus}.

Let us first recall some terminology about divisibility relations in monoids. If $\MM$ is a (left)-cancellative monoid and $\ff, \gg$ belong to~$\MM$, we say that $\ff$ \emph{left-divides}~$\gg$ or, equivalently, that $\gg$ is a \emph{right-multiple} of~$\gg$, denoted $\ff \dive \gg$, if $\ff \gg' = \gg$ holds for some~$\gg'$ in~$\MM$. If $1$ is the only invertible element in~$\MM$, the relation~$\dive$ is a partial ordering on~$\MM$. We say that $\hh$ is a \emph{least common right-multiple}, or \emph{right-lcm}, of~$\ff$ and~$\gg$ if $\hh$ is a least upper bound of~$\ff$ and~$\gg$ with respect to~$\dive$. Provided $1$ is the only invertible element in~$\MM$, the right-lcm is unique when it exists. If $\ff$ and~$\gg$ admit a right-lcm, the \emph{right-complement}~$\ff \under\gg$ of~$\ff$ in~$\gg$ is the unique element~$\gg'$ such that $\ff\gg'$ is the right-lcm of~$\ff$ and~$\gg$. As mentioned in Example~\ref{X:RC}, $\under$ obeys the RC-law (whenever defined). Of course, symmetric counterparts involve the right-divisibility relation, where $\ff$ \emph{right-divides}~$\gg$ if $\gg = \gg' \ff$ holds for some~$\gg'$.

If $\MM$ is a monoid, a \emph{length function} on~$\MM$ is a map $\lambda: \MM \to \NNNN$ such that $\gg \not= 1$ implies $\lambda(\gg) \ge 1$ and $\lambda(\gg\hh) = \lambda(\gg) + \lambda(\hh)$ always holds (we say \emph{pseudolength} when $=$ is relaxed into~$\ge$ in the last formula). Note that the existence of a (pseudo)-length function on~$\MM$ implies that $1$ is the only invertible element in~$\MM$.

Our first (new) observation is that the connection between the RC-law and the so-called cube condition in Garside theory makes it extremely easy to establish the basic properties of the monoid associated with an RC-quasigroup.

\begin{prop}
\label{P:Cancel}
Assume that $(\SS, \op)$ is an RC-quasigroup and $\MM$ is the associated monoid. Then the monoid~$\MM$ admits a length function, it is left-cancellative, and any two elements of~$\MM$ admit a unique right-lcm and a unique left-gcd. Moreover, $(\SS, \op)$ can be retrieved from~$\MM$: the set~$\SS$ is the set of atoms of~$\MM$ and, for $\ss \not= \tt$, the value of~$\ss\op\tt$ is the right-complement~$\ss\under\tt$ in~$\MM$ and the value of~$\ss\op\ss$ is the unique element of~$\SS \setminus \{\ss\under\tt \mid \tt\neq\ss\in\SS\}$.
\end{prop}

\begin{proof}
The relations of~\eqref{E:Str} preserve the length, so the length of $\SS$-words induces a length function on~$\MM$. 

Next, the presentation~\eqref{E:Str} contains exactly one relation of the form $\ss... = \tt...$ for each pair of generators~$\ss, \tt$ in~$\SS$. There exists for such presentations, which are called \emph{right-complemented}, a general approach that enables one to easily establish properties of the associated monoid. Assume we consider a monoid generated by a set~$\SS$ and relations of the form $\ss \RC(\ss, \tt) = \tt \RC(\tt, \ss)$ where $\RC$ maps~$\SS \times \SS$ to~$\SS$. Then \cite[Prop.~6.1 and~6.9]{Dgp} (or \cite[Prop.~II.4.16]{Garside}) says that, when the ``cube condition''
\begin{equation}
\label{E:Cube}
\RC(\RC(\rr, \ss), \RC(\rr, \tt)) = \RC(\RC(\ss, \rr), \RC(\ss, \tt))
\end{equation}
holds for all~$\rr, \ss, \tt$ in~$\SS$, the involved monoid is left-cancellative, any two of its elements admit a right-lcm, and the right-lcm of distinct elements~$\ss, \tt$ of~$\SS$ is $\ss (\ss \op \tt)$ (and $\tt (\tt \op \ss)$). In the current case of~$\MM$, the map~$\RC$ coincides with the operation~$\op$, and the assumption that $(\SS, \op)$ obeys the RC-law guarantees that \eqref{E:Cube} is satisfied. Hence $\MM$ is left-cancellative and any two elements of~$\MM$ admit a right-lcm. Standard general arguments then show that any two elements also admit a left-gcd, that is, a greatest lower bound with respect to the left-divisibility relation.

Finally, as there is no relation involving a word of length one in~\eqref{E:Str}, the elements of~$\SS$ are atoms, and every element not lying in~$\SS \cup \{1\}$ is not an atom. So $\SS$ is the atom set of~$\MM$. Next, for distinct $\ss, \tt$ in~$\SS$, the right-lcm of~$\ss$ and~$\tt$ is $\ss (\ss\op\tt)$, so, by definition, $\ss\under\tt$ is equal to~$\ss\op\tt$. Thus all nondiagonal values~$\ss\op\tt$ can be retrieved from~$\MM$. Finally, all left-translations of $(\SS, \op)$ are one-to-one, so $\ss\op\ss$ must be the unique element of $\SS \setminus \{\ss\op\tt \mid \ss, \tt\in \SS, \ss\neq\tt\}$, that is, of $\SS \setminus \{\ss\under\tt \mid \tt\neq\ss\in\SS\}$.
\end{proof}

\begin{coro}
\label{C:Ore}
Assume that $(\SS, \op)$ is a bijective RC-quasigroup and $\MM, \GG$ are the associated structure monoid and group. Then $\MM$ admits unique left- and right-lcms and left- and right-gcds, it is a Ore monoid, and $\GG$ is a group of left- and right-fractions for~$\MM$; this group is torsion-free. 
\end{coro}

\begin{proof}
Proposition~\ref{P:Cancel} guarantees left-cancellativity and existence of right-lcms and left-gcds in~$\MM$. Now, by Lemma~\ref{L:SecondOp}, there exists a second operation~$\opL$ on~$\SS$ such that $(\SS, \op, \opL)$ is an RLC-quasigroup. By construction, the presentation
\begin{equation}
\label{E:LStr}
\PRESp\SS{\{(\ss \opL \tt) \tt = (\tt \opL \ss) \ss \mid \ss \not= \tt \in \SS\}},
\end{equation}
coincides with the one of~\eqref{E:Str}. Proposition~\ref{P:Cancel} applied to the opposed monoid~$\MM\opp$ and to the RC-quasigroup~$(\SS, \opL\opp)$ implies that $\MM\opp$ is left-cancellative and admits right-lcms and left-gcds, hence $\MM$ is right-cancellative and admits left-lcms and right-gcds. Hence $\MM$ is in particular a Ore monoid (that is, a cancellative monoid where any two elements admit common left- and right-multiples). By a classical theorem of Ore~\cite{ClP}, its enveloping group~$\GG$, which admits as a group the presentation~\eqref{E:Str}, is a group of left- and right-fractions for~$\MM$. It is then known~\cite{Dha} that the group of fractions of a torsion-free monoid is torsion-free. 
\end{proof}

We can now use the RC-calculus of Section~\ref{S:RCCalculus} to easily describe the right-lcms of atoms in terms of the ``polynomials''~$\SPol_\nn$, thus avoiding the developments of~\cite{Cho}.

\begin{lemm}
\label{L:Lcm}
Assume that $(\SS, \op)$ is an RC-quasigroup and $\MM$ is the associated monoid. Then elements $\ss_1 \wdots \ss_\nn$ of~$\SS$ are pairwise distinct if and only if the element $\SPol_\nn(\ss_1 \wdots \ss_\nn)$ is the right-lcm of~$\ss_1 \wdots \ss_\nn$ in~$\MM$. In this case, if, in addition, $(\SS, \op)$ is bijective, $\SPol_\nn(\ss_1 \wdots \ss_\nn)$ is also the left-lcm of the elements $\sst_1 \wdots \sst_\nn$ defined by $\sst_\ii = \Pol_\nn(\ss_1 \wdots \widehat{\ss_\ii} \wdots \ss_\nn, \ss_\ii)$.
\end{lemm}

\begin{proof}
Assume that $\ss_1 \wdots \ss_\nn$ are pairwise distinct in~$\SS$. Let $\Pol'_\nn$ and~$\SPol'_\nn$ be the counterparts of~$\Pol_\nn$ and~$\SPol_\nn$ respectively where the right-complement operation~$\under$ replaces~$\op$. We first prove using induction on~$\ii$ the equality 
\begin{equation}
\label{E:Lcm}
\Pol_\ii(\ss_{\pi(1)} \wdots \ss_{\pi(\ii)}) = \Pol'_\ii(\ss_{\pi(1)} \wdots \ss_{\pi(\ii)})
\end{equation}
for every~$\ii$ and every permutation~$\pi$ in~$\Sym_\ii$. For $\ii = 1$, we have $\Pol_1(\ss_{\pi(1)}) = \ss_{\pi(1)} = \Pol'_1(\ss_{\pi(1)})$, and the result is straightforward. Assume $\ii \ge 2$. Put $\ss = \Pol_\ii(\ss_{\pi(1)} \wdots \ss_{\pi(\ii)})$, $\ss' = \Pol_\ii(\ss_{\pi(1)} \wdots \ss_{\pi(\ii-2)}, \ss_{\pi(\ii)}, \ss_{\pi(\ii-1)})$, $\tt = \Pol_\ii(\ss_{\pi(1)} \wdots \ss_{\pi(\ii)})$, and $\tt' = \Pol_{\ii-1}(\ss_{\pi(1)} \wdots \ss_{\pi(\ii-2)}, \ss_{\pi(\ii)})$. By definition of~$\Pol_\ii$ from~$\Pol_{\ii-1}$, we have $\ss = \tt \op \tt'$ and $\ss' = \tt' \op \tt$. By Lemma~\ref{L:Distinct} applied to $(\ss_{\pi(1)} \wdots \ss_{\pi(\ii)})$, the assumption $\ss_{\pi(\ii-1)} \not= \ss_{\pi(\ii)}$ implies $\ss \not= \ss'$, which implies $\tt \op \tt' \not= \tt' \op \tt$. By Proposition~\ref{P:Cancel}, the latter relation implies $\tt \op \tt' = \tt \under \tt'$ and $\tt'\op\tt = \tt' \under \tt$ in~$\MM$. The induction hypothesis implies $\tt = \Pol'_\ii(\ss_{\pi(1)} \wdots \ss_{\pi(\ii)})$ and $\tt' = \Pol'_{\ii-1}(\ss_{\pi(1)} \wdots \ss_{\pi(\ii-2)}, \ss_{\pi(\ii)})$, so we deduce $\ss = \tt \under \tt' = (\Pol'_\ii(\ss_{\pi(1)} \wdots \ss_{\pi(\ii)})) \under (\Pol'_{\ii-1}(\ss_{\pi(1)} \wdots \ss_{\pi(\ii-2)}, \ss_{\pi(\ii)}))$, that is, $\ss = \Pol'_\ii(\ss_{\pi(1)} \wdots \ss_{\pi(\ii)})$. Then \eqref{E:Lcm} implies $\SPol_\nn(\ss_1 \wdots \ss_\nn) = \SPol'_\nn(\ss_1 \wdots \ss_\nn)$, meaning that $\SPol_\nn(\ss_1 \wdots \ss_\nn)$ is the right-lcm of~$\ss_1 \wdots \ss_\nn$ since a trivial induction shows that $\SPol'_\nn(\ss_1 \wdots \ss_\nn)$ is the right-lcm of~$\ss_1 \wdots \ss_\nn$ for every~$\nn$. 

For the other direction, let $\nn'$ be the cardinal of $\{\ss_1 \wdots \ss_\nn\}$. The result above implies that, if $\II$ is a cardinal~$\nn'$ subset of~$\SS$, then the right-lcm~$\Delta_\II$ of~$\II$ has length~$\nn'$ in~$\MM$. So, if $\nn' < \nn$ holds, the right-lcm of $\{\ss_1 \wdots \ss_\nn\}$ is an element of~$\MM$ that has length~$\nn'$, and it cannot be $\SPol_\nn(\ss_1 \wdots \ss_\nn)$ which, by definition, has length~$\nn$. 

Finally, assume that $(\SS, \op)$ is bijective and $\ss_1 \wdots \ss_\nn$ are pairwise distinct. Let~$\opL$ be the second operation provided by Lemma~\ref{L:SecondOp}. Then $(\SS, \op, \opL)$ is an RLC-quasigroup. By Lemma~\ref{L:Distinct}, $\sst_1 \wdots \sst_\nn$ are pairwise distinct. Then $(\SS, \opL)$ is an LC-quasigroup, so the counterpart of the above results implies that $\SPolt_\nn(\sst_1 \wdots \sst_\nn)$ is a left-lcm of $\sst_1 \wdots \sst_\nn$ in~$\MM$. Now, by Lemma~\ref{L:IterInvol2}, $\SPolt_\nn(\sst_1 \wdots \sst_\nn)$ is equal to~$\SPol_\nn(\ss_1 \wdots \ss_\nn)$.
\end{proof}

A \emph{Garside family}~\cite{Dif} in a monoid~$\MM$ is a generating family~$\Sigma$ such that every element~$\gg$ of~$\MM$ admits a (unique) $\Sigma$-normal decomposition, meaning a sequence $(\ss_1 \wdots \ss_\pp)$ such that $\gg = \ss_1 \pdots \ss_\pp$ holds and $\ss_\ii$ is the greatest left-divisor of~$\ss_\ii \pdots \ss_\pp$ lying in~$\Sigma$ for every~$\ii$. 

\begin{prop}
\label{P:GarFam}
Assume that $(\SS, \op)$ is a bijective RC-quasigroup and $\MM, \GG$ are the associated monoid and group. Then the right-lcm~$\Delta_\II$ of a cardinal~$\nn$ subset~$\II$ of~$\SS$ belongs to~$\SS^\nn$, it is the left-lcm of (another) cardinal~$\nn$ subset of~$\SS$, the map $\II \mapsto \Delta_\II$ is injective, and its image is the smallest Garside family of~$\MM$ that contains~$1$. 
\end{prop}

\begin{proof}
By Proposition~\ref{P:Cancel}, $\MM$ admits a length function, it is left-cancellative, and that any two elements of~$\MM$ admit a right-lcm. Hence, by~\cite[Prop.~3.25]{Dif} (or \cite[Prop.~IV.2.46]{Garside}), $\MM$ admits a smallest Garside family~$\Sigma$, namely the closure of the atoms, that is, of~$\SS$, under the right-lcm and right-complement operations. We claim that $\Sigma$ is the closure~$\Sigma'$ of~$\SS$ under the sole right-lcm operation. 

By definition, $\Sigma'$ is included in~$\Sigma$, and the point is to prove that $\Sigma'$ is closed under the right-complement operation. This follows from the formula
\begin{equation}
\label{E:IterLcm}
\ff \under \LCM{\gg_1}{..., \gg_\nn} = \LCM{\ff\under\gg_1}{..., \ff\under\gg_\nn},
\end{equation}
which is valid in every monoid that admits unique right-lcms as shows an easy induction on~$\nn$. So assume that $\gg$ belongs to~$\Sigma'$, that is, $\gg$ is a right-lcm of elements $\tt_1 \wdots \tt_\nn$ of~$\SS$. If $\ff$ lies in~$\SS$, then, for every~$\ii$, the element $\ff\under\tt_\ii$ belongs to~$\SS \cup \{1\}$ since it is either $\ff\op\tt_\ii$, if $\ff$ and $\tt_\ii$ are distinct, or~$1$, if $\ff$ and $\tt_\ii$ coincide. Then \eqref{E:IterLcm} shows that $\ff\under\gg$ belongs to~$\Sigma'$ for every~$\ff$ in~$\SS$. Using induction on the length of~$\ff$, we deduce a similar result for every~$\ff$ in~$\MM$ from the formula $(\ff_1 \ff_2) \under \gg = \ff_2 \under (\ff_1 \under \gg)$. So $\Sigma'$ is closed under~$\under$ and it coincides with~$\Sigma$.

For~$\II$ a finite subset of~$\SS$, write~$\Delta_\II$ for the right-lcm of~$\II$. Lemma~\ref{L:Lcm} shows that $\Delta_\II$ has length~$\card\II$. Now, assume that $\II, \JJ$ are finite subsets of~$\SS$ and $\Delta_\II = \Delta_\JJ$ holds. As every element of~$\II \cup \JJ$ left-divides~$\Delta_\II$, we must have $\Delta_{\II\cup\JJ} = \Delta_\II = \Delta_\JJ$. This implies $\card{(\II \cup \JJ)} = \card\II = \card\JJ$, whence $\II = \II\cup\JJ = \JJ$.  So the map $\II \mapsto \Delta_\II$ is a bijection of~$\Pwfin(\SS)$ to~$\Sigma$.

Finally, if $\ss_1 \wdots \ss_\nn$ are pairwise distinct elements of~$\SS$, then $\SPol_\nn(\ss_1 \wdots \ss_\nn)$ is the right-lcm of $\ss_1 \wdots \ss_\nn$ and, by Lemma~\ref{L:Lcm}, the latter is also the left-lcm of the elements $\sst_1 \wdots \sst_\nn$ defined by $\sst_\ii = \Pol_\nn(\ss_1 \wdots \widehat{\ss_\ii} \wdots \ss_\nn, \ss_\ii)$.
\end{proof}

We turn to Garside monoids~\cite{Dgk}. A pair $(\MM, \Delta)$ is a \emph{Garside monoid} if $\MM$ is a cancellative monoid, it admits a weak length function, every two elements admit left- and right-lcms and gcds, and $\Delta$ is a Garside element in~$\MM$, meaning that the left- and right-divisors of~$\Delta$ coincide, generate~$\MM$, and are finite in number. We often say that a monoid~$\MM$ is a Garside monoid if there exists~$\Delta$ in~$\MM$ such that $(\MM, \Delta)$ satisfies the above conditions. 

\begin{prop}
\label{P:GarMon}
Assume that $(\SS, \op)$ is a bijective RC-quasigroup of cardinal~$\nn$ and $\MM$ is the associated monoid. Then the right-lcm~$\Delta$ of~$\SS$ is a Garside element in~$\MM$, it admits $2^\nn$ (left- or right-) divisors, and $(\MM, \Delta)$ is a Garside monoid.
\end{prop}

\begin{proof}
As above, write $\Delta_\II$ for the right-lcm of~$\II$ for~$\II \subseteq \SS$, and write $\Delta$ for~$\Delta_\SS$. By Proposition~\ref{P:GarFam}, the family~$\Sigma$ of all elements~$\Delta_\II$ is the smallest Garside family containing~$1$ in~$\MM$, and it has $2^\nn$~elements. By definition, $\Delta_\II$ left-divides~$\Delta_\SS$, that is, every element of~$\Sigma$ left-divides~$\Delta$, and, moreover, $\Delta$ lies in~$\Sigma$. This means that the Garside family~$\Sigma$ is what is called right-bounded by~$\Delta$ \cite[Def.~VI.1.1]{Garside}, and $\Delta$ is a right-Garside element in~$\MM$. Moreover, by Corollary~\ref{C:Ore}, $\MM$ is also right-cancellative, and $\Sigma$ is a finite subset of~$\MM$ that generates~$\MM$ and is closed under the right-complement operation. By~\cite[Prop.~2.1]{Dgk}, $\Delta$ is a Garside element in~$\MM$, and $(\MM, \Delta)$ is a Garside monoid. 
\end{proof}

\begin{exam}
\label{X:Lattice}
\rightskip 35 mm
Let $\SS = \{\tta, \ttb, \ttc\}$, and let $\op$ be determined by~$\xx\op\yy = \ff(\yy)$ where $\ff$ is the cycle $\aa \mapsto \bb \mapsto \cc \mapsto \aa$. Then, as seen in Example~\ref{X:RC}, $(\SS, \op)$ is a bijective RC-quasigroup, and it is eligible for the above results. The associated monoid admits the presentation 
$$\PRESp{\tta, \ttb, \ttc}{\tta\ttc = \ttb^2, \tta^2 = \ttc\ttb, \ttb\tta = \ttc^2}.\HS{24}$$
The right-lcm~$\Delta$ of~$\SS$ is then~$\tta^3$, which is also $\ttb^3$ and~$\ttc^3$, and the lattice of the $8$~divisors of~$\Delta$ is shown on the right. 
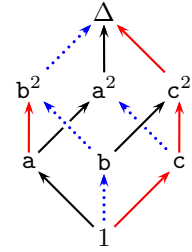
\hfill
\begin{picture}(0,0)(-10,-3)
\rput(10,0){\rnode[c]{1}{$1$}}
\rput(0,10){\rnode[c]{a}{$\tta$}}
\rput(10,10){\rnode[c]{b}{$\ttb$}}
\rput(20,10){\rnode[c]{c}{$\ttc$}}
\rput(0,20){\rnode[c]{bb}{$\ttb^2$}}
\rput(10,20){\rnode[c]{aa}{$\tta^2$}}
\rput(20,20){\rnode[c]{cc}{$\ttc^2$}}
\rput(10,30){\rnode[c]{aaa}{$\Delta$}}
\psset{nodesep=0.5mm}
\AArrow(a,aa)
\AArrow(1,a)
\BArrow(1,b)
\CArrow(1,c)
\CArrow(a,bb)
\BArrow(b,bb)
\BArrow(c,aa)
\AArrow(b,cc)
\CArrow(c,cc)
\AArrow(aa,aaa)
\BArrow(bb,aaa)
\CArrow(cc,aaa)
\end{picture}
\end{exam}

We conclude the section with a result in the other direction. If $(\SS, \sol)$ is a solution of the YBE or, equivalently, if $(\SS, \op)$ is a finite bijective RC-quasigroup, then the associated monoid is a Garside monoid and, moreover, its definition implies that the latter admits a presentation of a certain form. We see now is that, conversely, every Garside monoid with the above properties is associated with a (bijective) RC-quasigroup. Once again, a YBE version appears in~\cite{Cho} and the point here is to show that using the RC-law provides alternative and hopefully more simple arguments. 

\begin{prop}
\label{P:Charac}
Assume that $\MM$ is a monoid with atom set~$\SS$ of cardinal~$\nn$. Then the following are equivalent:

\ITEM1 There exists an operation~$\op$ such that $(\SS, \op)$ is a bijective RC-quasigroup and $\MM$ is associated with~$(\SS, \op)$---or, equivalently, there exists $\sol$ such that $(\SS, \sol)$ is a solution of YBE and $\MM$ is associated with~$(\SS, \sol)$;

\ITEM2 The monoid~$\MM$ is a Garside monoid admitting a presentation in terms of~$\SS$ that consists of one relation $\ss \theta(\ss, \tt) = \tt \theta(\tt, \ss)$ for $\ss \not= \tt$ in~$\SS$ with $\theta : \SS \times \SS \to \SS$ such that $\tt \mapsto \theta(\ss, \tt)$ is injective for every~$\ss$.

\ITEM3 The monoid~$\MM$ is a Garside monoid admitting a presentation in terms of~$\SS$ consisting of $\nn\choose2$ relations~$\uu = \vv$ with $\uu, \vv$ of length two such that every length two $\SS$-word appears in at most one relation.
\end{prop}

It follows from Definition~\ref{D:StrYB} and Proposition~\ref{P:GarMon} that \ITEM1 implies~\ITEM2; on the other hand, a presentation satisfying the conditions of~\ITEM2 necessarily satisfies those of~\ITEM3 since a length two $\SS$-word~$\ss\ss'$ may appear in a relation $\ss... = \tt...$ only if $\ss' = \theta(\ss, \tt)$ holds, which happens for at most one~$\tt$. So we are left with proving that \ITEM3 implies~\ITEM2 and that \ITEM2 implies~\ITEM1.

\begin{proof}[Proof of \ITEM3$\Rightarrow$\ITEM2 in Proposition~\ref{P:Charac}]
Let $\RR$ be the considered list of relations. Assume that $\ss, \tt$ are distinct elements of~$\SS$ and $\RR$ contains at least two relations $\ss... = \tt...$, say $\ss\tt' = \tt\ss'$ and $\ss\tt'' = \tt\ss''$ with $(\ss', \tt') \not= (\ss'', \tt'')$. As $\MM$ is cancellative, we have $\ss\tt' \not= \ss\tt''$, so $\ss\tt'$ and $\ss\tt''$ are two common right-multiples of~$\ss$ and~$\tt$ of length~$2$: this contradicts the existence of a right-lcm for~$\ss$ and~$\tt$, as the latter can have neither length~$1$ nor length~$2$. Hence $\RR$ contains at most one relation $\ss... = \tt...$ for $\ss \not= \tt$. On the other hand, $\RR$ contains no relation $\ss... = \ss...$ since $\MM$ is left-cancellative and $\ss\tt = \ss\tt'$ would imply $\tt = \tt'$. As there are $\nn\choose2$ pairs of distinct elements of~$\SS$, we deduce that $\RR$ contains exactly one relation $\ss... = \tt...$ for~$\ss \not= \tt$ in~$\SS$. We can then write the latter relation as $\ss \theta(\ss, \tt) = \tt \theta(\tt, \ss)$ for some map~$\theta$ from~$\SS^2$ to~$\SS$. The assumption that no word appears in two relations then implies that every map $\tt \mapsto \theta(\ss, \tt)$ is injective.
\end{proof}

We recall that, if $\MM$ is a cancellative monoid in which any two elements admit a unique right-lcm, then, for $\ff, \gg$ in~$\MM$, the right-complement of~$\ff$ in~$\gg$ is the element~$\ff \under\gg$ such that $\ff(\ff \under\gg)$ is the right-lcm of~$\ff$ and~$\gg$. Under symmetric assumptions, the \emph{left-commplement} $\ff / \gg$ is the element such that $(\ff / \gg)\gg$ is the left-lcm of~$\ff$ and~$\gg$. The operation~$\under$ obeys the RC-law, whereas $/$ obeys the LC-law.  

\begin{proof}[Proof of \ITEM2$\Rightarrow$\ITEM1 in Proposition~\ref{P:Charac}]
By assumption, $\ss\theta(\ss, \tt)$ is the right-lcm of~$\ss$ and~$\tt$ (which exists since $\MM$ is a Garside monoid) for $\ss \not= \tt$ in~$\SS$, so $\ss \under \tt = \theta(\ss, \tt)$ holds. We define a binary operation~$\op$ on~$\SS$. First, we put $\ss\op\tt = \ss\under\tt$ for $\ss \not= \tt$. As $\tt \not= \tt'$ implies $\theta(\ss, \tt) \not= \theta(\ss, \tt')$, the map $\xx \mapsto \ss\op\xx$ is injective on~$\SS \setminus\{\ss\}$, so the complement of $\{\ss\op\xx \mid \xx\not=\ss\}$ in~$\SS$ consists of a unique element, which we define to be $\ss\op\ss$. We thus obtain an operation~$\op$ whose left-translations are one-to-one. 

Next, we know that \ITEM2 implies \ITEM3. Now, in \ITEM3, the left and the right sides play symmetric roles. Hence the symmetric counterpart of~\ITEM2 is true and, by (the counterpart of) the above argument, we obtain a binary operation~$\opL$ on~$\SS$ connected with the left-complement operation~$/$ and such that the right-translations~$\opL$ are one-to-one. We claim that $\op$ and $\opL$ satisfy~\eqref{E:InvolRLC}. First, assume $\ss \not= \tt$. Then $\ss(\ss\op\tt) = \tt(\tt\op\ss)$ lies in~$\RR$ (the considered set of relations), implying $\ss\op\tt \not= \tt\op\ss$. Next, $\ss(\ss\op\tt)$ is the right-lcm of~$\ss$ and~$\tt$, and, as $\ss\op\tt$ and $\tt\op\ss$ are distinct, $\ss(\ss\op\tt)$ is the left-lcm of~$\ss\op\tt$ and~$\tt\op\ss$, which implies $(\ss\op\tt)\opL(\tt\op\ss) = \tt$. Now, put $\ss' = \ss\op\ss$ and $\rr = \ss'\opL\ss'$. For $\tt \not= \ss$, we have $\ss\op\tt \not= \ss'$, whence $\rr = \ss'/(\ss\under\tt)$. Then $\rr (\ss\under\tt) = ((\ss\under\tt)/\ss')\ss'$ is a relation of~$\RR$, which implies $(\ss\under\tt)/\ss' \not= \ss$ since $\RR$ contains no relation $\ss\ss' = ...$. Since $(\ss\under\tt)/\ss' \not= \ss$ holds for every~$\tt$ distinct of~$\ss$, we deduce $\ss'\opL\ss' = \ss$ since $\ss'\opL\ss'$ is the only element of~$\SS$ that is not of the form $(\ss\under\tt)\opL\ss'$ with~$\tt\not=\ss$. Hence $(\ss\op\ss)\opL(\ss\op\ss) = \ss$ holds, and the first involutivity law is satisfied in~$(\SS, \op, \opL)$. By a symmetric argument, the second involutivity law holds as well.

Finally, we claim that $(\SS, \op)$ satisfies the RC-law. Let $\rr, \ss, \tt$ lie in~$\SS$. Assume first that $\rr, \ss, \tt$ are pairwise distinct. Then we have $\rr\op\ss \not= \rr\op\tt$ and $\ss\op\rr \not= \ss\op\tt$, whence $(\rr\op\ss)\op(\rr\op\tt) = (\rr\under\ss)\under(\rr\under\tt) = (\ss\under\rr)\under(\ss\under\tt) = (\ss\op\rr)\op(\ss\op\tt)$, since the right-complement operation~$\under$ satisfies the RC-law. Next, for $\rr = \ss$, the RC-law tautologically holds for $\rr, \ss, \tt$. So there only remain the cases when $\rr \not= \ss$ and $\tt$ is either~$\rr$ or~$\ss$, that is, we have to establish $(\rr\op\ss)\op(\rr\op\ss) = (\ss\op\rr)\op(\ss\op\ss)$ and $(\ss\op\rr)\op(\ss\op\rr) = (\rr\op\ss)\op(\rr\op\rr)$, that is, owing to $\rr\not=\ss$, 
\begin{equation}
\label{E:GarRC}
(\rr\under\ss)\op(\rr\under\ss) = (\ss\under\rr)\under(\ss\op\ss) \text{ and } (\ss\under\rr)\op(\ss\under\rr) = (\rr\under\ss)\under(\rr\op\rr).
\end{equation}
Assume $\zz \not= \rr, \ss$ and put $\zz' = (\rr\under\ss)\under(\rr\under\zz)$, which is also $\zz' = (\ss\under\rr)\under(\ss\under\zz)$ since $\under$ satisfies the RC-law. Then we have $\rr\under\zz \not= \rr\under\ss$, whence $\zz' \not= (\rr\under\ss)\op(\rr\under\ss)$. Also, we have $\ss\under\zz \not= \ss\op\ss$, whence $\zz' \not= (\ss\under\rr)\under(\ss\op\ss)$. Arguing similarly with~$\rr$ and~$\ss$ exchanged, we find $\zz' \not= (\ss\under\rr)\op(\ss\under\rr)$ and $\zz' \not= (\rr\under\ss)\under(\rr\op\rr)$. So, $\zz'$ is distinct from the four expressions occurring in~\eqref{E:GarRC} and, therefore, the only possible values for the latter are the two elements of~$\SS$ that are not of the form $(\rr\under\ss)\under(\rr\under\zz)$ with $\zz \not= \rr, \ss$. As left-translations of~$\op$ are injective, we must have $(\rr\under\ss)\op(\rr\under\ss) \not= (\rr\under\ss)\under(\rr\op\rr)$ and $(\ss\under\rr)\under(\ss\op\ss) \not= (\ss\under\rr)\op(\ss\under\rr)$. So, in order to establish \eqref{E:GarRC}, it suffices to show that $(\rr\under\ss)\op(\rr\under\ss) = (\ss\under\rr)\op(\ss\under\rr)$ is impossible. Now $\rr\not= \ss$ implies $\rr\op\ss \not= \ss\op\ss$, so it is enough to prove that $\xx \not= \yy$ implies $\xx\op\xx \not= \yy\op\yy$: this follows from the involutivity relation $(\xx\op\xx) \opL (\xx\op\xx) = \xx$ established above.

So $(\SS, \op)$ is an RC-quasigroup. By~\cite[Theorem~2]{RumYB}, it is necessarily bijective since $\SS$ is finite---alternatively, one can check that $(\SS, \op, \opL)$ is an RLC-quasigroup and $(\SS, \op)$ is bijective by Lemma~\ref{L:SecondOp}. By construction, $\MM$ admits the presentation $\Pres\SS\RR$, so it is (isomorphic to) the monoid associated with~$(\SS, \op)$.
\end{proof}

\begin{rema}
The injectivity condition for left-translations associated with~$\theta$ is necessary in Proposition~\ref{P:Charac}\ITEM2: the dual braid monoid $\PRESp{\tta, \ttb, \ttc}{\tta\ttb = \ttb\ttc = \ttc\tta}$ is a Garside monoid that admits a presentation of the considered type but it is not associated with an RC-quasigroup since (for instance) the right-lcm of the atoms has $5$~divisors. Now, in this case, we have $\theta(\tta, \ttb) = \theta(\tta, \ttc) = \ttb$.
\end{rema}

\section{The $I$-structure}
\label{S:IStructure}

It has been known since~\cite{Gav} and~\cite{JeO} that the monoids associated with solutions of YBE admit a nice geometric characterization as those monoids that admit an $I$-structure, meaning that their Cayley graph is a twisted copy of that of a free abelian monoid. These results can be easily established using the framework of RC-quasigroups and the computational formulas of Section~\ref{S:RCCalculus}. In particular, we shall see that the $I$-structure can be explicitly determined using the polynomials~$\SPol_\nn$. 

In view of the final results in Section~\ref{S:Coxeter}, it is convenient to start with a slightly extended definition in which the reference monoid need not be free Abelian.

\begin{defi}
\label{D:IStr}
Assume that $\MMb$ and $\MM$ are monoids generated by a set~$\SS$. A \emph{right-$I$-structure of shape~$(\MMb, \SS)$} for~$\MM$ is a bijective map~$\nu: \MMb \to \MM$ satisfying $\nu(1) = 1$, $\nu(\ss) = \ss$ for~$\ss$ in~$\SS$ and, for every~$\gga$ in~$\MMb$,
\begin{equation}
\label{E:IStr}
\{\nu(\gga\ss) \mid \ss \in \SS\} = \{\nu(\gga)\ss \mid \ss \in \SS\}.
\end{equation}
\end{defi}

When the reference monoid~$\MMb$ is the $\SS$-power of a monoid, we skip~$\SS$ and just say ``right-$I$-structure of shape~$\MMb$''; if $\MMb$ is a free Abelian monoid, we also skip~$\MMb$ and---as is usual---say ``right-$I$-structure''. A monoid is said to be \emph{of right-$I$-type}~\cite{Gav, JeO} if it admits a right $I$-structure.

The existence of a right-$I$-structure of shape~$(\MMb, \SS)$ on a monoid~$\MM$ provides a bijection from the Cayley graph of~$\MMb$ relative to~$\SS$ onto that of~$\MM$ that preserves the path length but changes labels. Note that \eqref{E:IStr} is equivalent to the existence, for every~$\gga$ in~$\MMb$, of a permutation~$\psi(\gga)$ of~$\SS$ satisfying, for every~$\ss$ in~$\SS$, 
\begin{equation}
\label{E:IStrBis}
\nu(\gga\ss) = \nu(\gga) \cdot \psi(\gga)(\ss).
\end{equation}
We shall use~$\psi$ with this meaning everywhere in the sequel.

If $\MMb$ is a free Abelian monoid based on~$\SS$ and, more generally, if every permutation of~$\SS$ induces an automorphism of~$\MMb$, then the condition $\nu(\ss) = \ss$ for~$\ss$ in~$\SS$, which amounts to $\psi(1)$ being the identity of~$\SS$, can be ensured by precomposing~$\nu$ with the automorphism induced by~$\psi(1)\inv$ and, therefore, it could be removed from the definition without changing the range of the latter.

Hereafter we denote by~$\NNNN^{(\SS)}$ the free Abelian monoid based on a (finite or infinite) set~$\SS$, identified with the set of all finite support $\SS$-indexed sequences of elements of~$\NNNN$ (we use $\NNNN^\SS$ when $\SS$ is finite); an element~$\ss$ of~$\SS$ is identified with the sequence whose only non-zero entry is the $\ss$-entry, which is~$1$.

We first establish the following explicit ``RC-version'' of the result of~\cite{Gav, JeO}:

\begin{prop}
\label{P:RCItype}
Assume that $(\SS, \op)$ is an RC-quasigroup and $\MM$ is the associated monoid. Then the map~$\nu$ defined from~$\op$ by $\nu(\ss_1 \pdots \ss_\nn) = \SPol_\nn(\ss_1 \wdots \ss_\nn)$ is a right $I$-structure on~$\MM$. 
\end{prop}

\begin{proof}
We first define a map~$\nu^*$ from the free monoid~$\SS^*$ based on~$\SS$ to~$\MM$ by $\nu^*(\ew) = 1$ and $\nu^*(\ss_1 \sep \pdots \sep \ss_\nn) = \SPol_\nn(\ss_1 \wdots \ss_\nn)$ for $\nn \ge 1$, where $\ss_1 \sep \pdots \sep \ss_\nn$ is the length~$\nn$ word with letters~$\ss_1 \wdots \ss_\nn$. By Lemma~\ref{L:SPol}, $\nu^*(\ss_1 \sep \pdots \sep \ss_\nn)$ does not depend on the order of the entries, so $\nu^*$ induces a well-defined map~$\nu$ from the free Abelian monoid~$\NNNN^{(\SS)}$ to~$\MM$. We claim that the latter provides the expected right-$I$-structure on~$\MM$. First, the equalities $\nu(1) = 1$ and $\nu(\ss) = \ss$ for~$\ss$ in~$\SS$ are obvious. Next, let $\gga$ belong to~$\NNNN^{(\SS)}$, say $\gga = \ss_1 \pdots \ss_\nn$. Then the definition of~$\nu$ gives $\nu(\gga\ss) = \SPol_{\nn+1}(\ss_1 \wdots \ss_\nn, \ss)$,  whence $\nu(\gga\ss) = \SPol_\nn(\ss_1 \wdots \ss_\nn) \cdot \Pol_{\nn+1}(\ss_1 \wdots \ss_\nn, \ss)$. By  Lemma~\ref{L:Distinct}\ITEM1, the map $\ss \mapsto \nu(\gga\ss)$ is a bijection of~$\SS$ into itself, hence \eqref{E:IStr} holds. 

It remains to show that $\nu$ is a bijection from~$\NNNN^{(\SS)}$ to~$\MM$. Let $\gg$ belong to~$\MM$, say $\gg = \ss_1 \pdots \ss_\nn$ with $\ss_1 \wdots \ss_\nn$ in~$\SS$. By Lemma~\ref{L:Distinct}\ITEM2, there exist $\rr_1 \wdots \rr_\nn$ in~$\SS$ satisfying $\Pol_\ii(\rr_1 \wdots \rr_\ii) = (\ss_1 \wdots \ss_\ii)$ for $1 \le \ii \le \nn$, whence $\SPol_\nn(\rr_1 \wdots \rr_\nn) = \ss_1 \pdots \ss_\nn = \gg$. By definition, this means that $\nu(\rr_1 \pdots \rr_\nn) = \gg$ holds, and $\nu$ is surjective.

Finally, assume that $\gga, \gga'$ belong to~$\NNNN^{(\SS)}$ and $\nu(\gga) = \nu(\gga')$ holds. As the elements of~$\MM$ have a well-defined length, the length of~$\gga$ and~$\gga'$ must be the same. Write $\gga = \rr_1 \pdots \rr_\nn$, $\gga' = \rr'_1 \pdots \rr'_\nn$ with $\rr_1 \wdots \rr'_\nn$ in~$\SS$. Define $\ss_\ii = \Pol_\ii(\rr_1 \wdots \rr_\ii)$ and $\ss'_\ii = \Pol_\ii(\rr'_1 \wdots \rr'_\ii)$. By definition, $\nu(\gga)$ is the class of $\ss_1 \sep \pdots \sep \ss_\nn$ in~$\MM$, and $\nu(\gga')$ is the class of $\ss'_1 \sep \pdots \sep \ss'_\nn$. The assumption $\nu(\gga) = \nu(\gga')$ means that these words are connected by a finite sequence of defining relations of~$\MM$. By Lemma~\ref{L:Distinct}, the map $(\xx_1 \wdots \xx_\nn) \mapsto (\Pol_1(\xx_1) \wdots \Pol_\nn(\xx_1 \wdots \xx_\nn) )$ of~$\Seq\SS\nn$ to itself is surjective, so we can assume without loss of generality that $\ss_1 \sep \pdots \sep \ss_\nn$ and $\ss'_1 \sep \pdots \sep \ss'_\nn$ are connected by one relation exactly, that is, there exist~$\ii$ satisfying $\ss_{\ii+1} = \ss_\ii\op\ss'_\ii$, $\ss'_{\ii+1} = \ss'_\ii\op\ss_\ii$ and $\ss'_\kk = \ss_\kk$  for $\kk \not= \ii, \ii+1$. The relations $\ss'_\kk = \ss_\kk$ inductively imply $\rr'_\kk = \rr_\kk$ for $\kk < \ii$. Next, writing $\vec\rr$ for $\rr_1 \wdots \rr_{\ii-1}$, we have $\ss_\ii = \Pol_\ii(\vec\rr, \rr_\ii)$ and $\ss'_\ii = \Pol_\ii(\vec\rr, \rr'_\ii)$. Then, we find 
$$\Pol_\ii(\vec\rr, \rr_\ii) \op \Pol_\ii(\vec\rr, \rr_{\ii+1}) = \Pol_{\ii+1}(\vec\rr, \rr_\ii, \rr_{\ii+1}) = \ss_{\ii+1} = \ss_\ii\op\ss'_\ii = \Pol_\ii(\vec\rr, \rr_\ii) \op \Pol_\ii(\vec\rr, \rr'_\ii).$$
As the left-translation by~$\Pol_\ii(\vec\rr, \rr_\ii)$ is injective, we deduce $\Pol_\ii(\vec\rr, \rr_{\ii+1}) =\nobreak \Pol_\ii(\vec\rr, \rr'_\ii)$, whence $\rr_{\ii+1} =\nobreak \rr'_\ii$ by Lemma~\ref{L:Distinct}\ITEM1. A symmetric argument gives $\rr'_{\ii+1} = \rr_\ii$. Finally, for $\kk > \ii+1$, the relations $\ss'_\kk = \ss_\kk$ inductively imply $\rr'_\kk = \rr_\kk$. Indeed, we have 
$$\Pol_\kk(\vec\rr, \rr_\ii, \rr'_\ii, \rr_{\ii+2} \wdots \rr_\kk) = \ss_\kk = \ss'_\kk = \Pol_\kk(\vec\rr, \rr'_\ii, \rr_\ii, \rr_{\ii+2} \wdots \rr'_\kk),$$
and, by Lemma~\ref{L:Pol}, switching the non-final entries~$\rr_\ii$ and~$\rr'_\ii$ in~$\Pol_\kk$ changes nothing, so $\rr'_\kk = \rr_\kk$ follows by Lemma~\ref{L:Distinct}. So the words $\rr_1 \sep \pdots \sep \rr_\nn$ and $\rr'_1 \sep \pdots \sep \rr'_\nn$ are obtained by switching two (adjacent) entries, hence they represent the same element in~$\NNNN^{(\SS)}$. Hence $\nu$ is injective, and it provides the expected right-$I$-structure on~$\MM$.
\end{proof}

Note that Proposition~\ref{P:RCItype} does not require the set~$\SS$ to be finite.

In the other direction, we prove that every finitely generated monoid of $I$-type is associated with an RC-quasigroup, an explicit version of the YBE results of~\cite{Gav, JeO}:

\begin{prop}
\label{P:IType}
If $\MM$ is a finitely generated monoid of right-$I$-type, there exists a unique finite RC-quasi\-group~$(\SS, \op)$ such that $\MM$ is the structure monoid of~$(\SS, \op)$: the set~$\SS$ is the atom set of~$\MM$ and $\op$ is determined by $\ss\op\tt = \ss\under\tt$ for $\ss\not=\tt$ and $\{\ss\op\ss\} = \SS \setminus \{\ss\under\tt \mid \tt\not=\nobreak\ss\}$. Moreover, the right-$I$-structure on~$\MM$ is unique: it is defined from the operation~$\op$ by $\nu(\ss_1 \pdots \ss_\nn) = \SPol_\nn(\ss_1 \wdots \ss_\nn)$. 
\end{prop}

Once again our proof relies on RC-calculus. One of the benefits is to obtain an argument that, although complete, in particular for establishing~\eqref{E:Product} below, is shorter than that of~\cite[Chapter~8]{JeOBook}. Another benefit is that Lemma~\ref{L:NuRC} extends to a nonfree Abelian monoid, which will be useful in Section~\ref{S:Coxeter}.

\begin{lemm}
\label{L:NuRC}
Assume that $\MM$ is a left-cancellative monoid and $\nu$ is a right-$I$-structure of shape~$(\MMb, \SS)$ on~$\MM$, where $\MMb$ is Abelian and  satisfies
\begin{equation}
\label{E:Condition}
\forall\ss, \tt, \ss', \tt' \in \SS\,  ((\ss \not= \tt \text{ and } \ss\tt' = \tt\ss') \Rightarrow (\ss' = \ss   \text{ and } \tt' = \tt)).
\end{equation}
Put $\ss \op \tt = \psi(\ss)(\tt)$ for~$\ss, \tt$ in~$\SS$. Then $(\SS, \op)$ is an RC-quasigroup.
\end{lemm}

\begin{proof}
We claim that, for all~$\ss, \tt, \ss', \tt'$ in~$\SS$ with $\ss \not= \tt$, the only equality $\ss\tt' = \tt\ss'$ holding in~$\MM$ is $\ss(\ss\op\tt) = \tt(\tt\op\ss)$. Indeed, by assumption, $\MMb$ is Abelian and, in~$\MM$, using~\eqref{E:IStrBis}, we obtain $\ss(\ss \op \tt) = \ss \psi(\ss)(\tt) = \nu(\ss\tt)  = \nu(\tt\ss) = \tt \psi(\tt)(\ss) = \tt(\tt \op \ss)$. On the other hand, assume $\ss \tt' = \tt \ss'$ in~$\MM$. Let $\tt'\refer = \psi(\ss)\inv(\tt')$ and $\ss'\refer = \psi(\tt)\inv(\ss')$. Always using~\eqref{E:IStrBis}, we find $\nu(\ss\tt'\refer) = \ss\tt' = \tt\ss' = \nu(\tt\ss'\refer)$ in~$\MM$, whence $\ss \tt'\refer = \tt\ss'\refer$ in~$\MMb$ since $\nu$ is bijective. The assumption on~$\MMb$ implies $\ss'\refer = \ss$ and $\tt'\refer = \tt$, whence $\ss' = \psi(\tt)(\ss) = \tt \op \ss$ and $\tt' = \psi(\ss)(\tt) = \ss \op \tt$. This establishes the claim.

Now, let $\gga$ belong to~$\MMb$ and $\ss, \tt$ belong to~$\SS$. Using~\eqref{E:IStrBis}, we find 
$$\nu(\gga\ss\tt) = \nu(\gga\ss) \cdot \psi(\gga\ss)(\tt) = \nu(\gga) \cdot \psi(\gga)(\ss) \cdot \psi(\gga\ss)(\tt),$$
and, similarly, $\nu(\gga\tt\ss) = \nu(\gga) \cdot \psi(\gga)(\tt) \cdot \psi(\gga\tt)(\ss)$. By assumption, $\MMb$ is Abelian, so we have $\gga\ss\tt = \gga\tt\ss$, whence $\nu(\gga\ss\tt) = \nu(\gga\tt\ss)$ and, merging the above expressions and left-cancelling~$\nu(\gga)$, which is legal as $\MM$ is left-cancellative, we find
\begin{equation}
\label{E:TheRel1}
\psi(\gga)(\ss) \cdot \psi(\gga\ss)(\tt) = \psi(\gga)(\tt) \cdot \psi(\gga\tt)(\ss).
\end{equation}
Assume first $\ss \not= \tt$. The elements $\psi(\gga)(\ss)$, $\psi(\gga\ss)(\tt)$, $\psi(\gga)(\tt)$, and $\psi(\gga\tt)(\ss)$ lie in~$\SS$, so, by the claim above, \eqref{E:TheRel1} implies
\begin{equation}
\label{E:TheRel2}
\psi(\gga\ss)(\tt) = \psi(\gga)(\ss)\, \op\, \psi(\gga)(\tt) \mbox{\quad and \quad } \psi(\gga\tt)(\ss) = \psi(\gga)(\tt)\, \op\, \psi(\gga)(\ss).
\end{equation}
When $\tt$ ranges over $\SS \setminus \{\ss\}$, the element $\psi(\gga)(\tt)$ ranges over $\SS \setminus \{\psi(\gga)(\ss)\}$, and $\psi(\gga)(\ss)\, \op\, \psi(\gga)(\tt)$ ranges over $\SS \setminus \{\psi(\gga)(\ss)\, \op\, \psi(\gga)(\ss)\}$. As $\psi(\gga\ss)$ is a bijection of~$\SS$, the only possibility is therefore $\psi(\gga\ss)(\ss) = \psi(\gga)(\ss)\, \op\, \psi(\gga)(\ss)$. Hence \eqref{E:TheRel2} is valid in~$\SS$ for all~$\gga$ in~$\MMb$ and~$\ss, \tt$ in~$\SS$.

Now, let $\rr$ lie in~$\SS$. Making $\gga = \rr$ in~\eqref{E:TheRel2} and applying the definition of~$\op$ gives $\psi(\rr\ss)(\tt) = (\rr\op\ss)\op(\rr\op\tt)$ and, similarly, $\psi(\ss\rr)(\tt) = (\ss\op\rr)\op(\ss\op\tt)$. Now, in~$\MMb$, we have $\rr\ss = \ss\rr$, whence $\psi(\rr\ss)(\tt) = \psi(\ss\rr)(\tt)$, and this gives $(\rr\op\ss)\op(\rr\op\tt) = (\ss\op\rr)\op(\ss\op\tt)$, the RC-law. So $(\SS, \op)$ is an RC-system. Moreover, by definition, $\psi(\ss)$ belongs to~$\Sym_\SS$, so the left-translations of~$\op$ are one-to-one, and $(\SS, \op)$ is an RC-quasigroup.
\end{proof}

\begin{lemm}
\label{E:NuRCBis}
Assume that $\MM$ is a left-cancellative monoid and $\nu$ is a right-$I$-structure of shape~$(\MMb, \SS)$ on~$\MM$, where $\MMb$ is Abelian and  satisfies~\eqref{E:Condition}. Then, for every~$\pp \ge 1$ and for all~$\ss_1 \wdots \ss_\pp$ in~$\SS$, we have
\begin{equation}
\label{E:Nu}
\psi(\ss_1 \pdots \ss_{\pp-1})(\ss_\pp) = \Pol_\pp(\ss_1 \wdots \ss_\pp)
\mbox{\quad and\quad} 
\nu(\ss_1 \pdots \ss_\pp) = \SPol_\pp(\ss_1 \wdots \ss_\pp),
\end{equation}
where $\Pol_\pp$ and $\SPol_\pp$ are associated with~$\op$ as in Section~\ref{S:RCCalculus} and values are taken in~$\MM$. 
Moreover, for all $\gga, \hha$ in~$\NNNN^{(\SS)}$, we have
\begin{equation}
\label{E:Product}
\nu(\gga\hha) = \nu(\gga)\nu(\psi(\gga)[\hha]) \text{\quad and\quad } \psi(\gga\hha) = \psi(\psi(\gga)[\hha]) \comp \psi(\gga)
\end{equation}
where $\psi(\gga)[\hha]$ is the result of applying $\psi(\gga)$ to~$\hha$ componentwise.
\end{lemm}

\begin{proof}
We begin with the left equality in~\eqref{E:Nu}, which is proved using induction on~$\pp$. For $\pp =\nobreak 1$, we have $\psi(1)(\ss_1) = \ss_1 = \Pol_1(\ss_1)$ and, for $\pp = 2$, we have $\psi(\ss_1)(\ss_2) = \ss_1 \op \ss_2 =\nobreak \Pol_2(\ss_1, \ss_2)$. For $\pp \ge 3$, using~\eqref{E:TheRel2}, the induction hypothesis, and the inductive definition of the monomials~$\Pol_\pp$, we find
\begin{multline*}
\HS{5}\psi(\ss_1 \pdots \ss_{\pp-1})(\ss_\pp) 
= \psi(\ss_1 \pdots \ss_{\pp-2})(\ss_{\pp-1}) \, \op \, \psi(\ss_1 \pdots \ss_{\pp-2})(\ss_\pp)\\ 
= \Pol_{\pp-1}(\ss_1 \wdots \ss_{\pp-1}) \, \op \, \Pol_{\pp-1}(\ss_1 \wdots \ss_{\pp-2}, \ss_\pp) 
= \Pol_\pp(\ss_1 \wdots \ss_\pp).\HS{5}
\end{multline*}
The right equality in~\eqref{E:Nu} then follows using a similar induction on~$\pp$. For $\pp =  1$, we have $\nu(\ss_1) = \ss_1 = \SPol(\ss_1)$. For $\pp \ge 2$, using~\eqref{E:IStrBis}, the left equality in~\eqref{E:Nu}, and the inductive definition of~$\SPol_\pp$, we find
\begin{multline*}
\HS{5}\nu(\ss_1 \pdots \ss_\pp) = \nu(\ss_1 \pdots \ss_{\pp-1}) \cdot \psi(\ss_1 \pdots \ss_{\pp-1})(\ss_\pp)\\ 
= \SPol_{\pp-1}(\ss_1 \wdots \ss_{\pp-1}) \cdot \Pol_\pp(\ss_1 \wdots \ss_\pp) = \SPol_\pp(\ss_1 \wdots \ss_\pp).\HS{5}
\end{multline*}

The definition of~$\Pol_\nn$ implies, for~Ê$\pp, \qq \ge 1$, the formal equality
\begin{equation*}
\label{E:PolId1}
\Pol_{\pp+\qq}(\vec\xx, \yy_1 \wdots \yy_\qq) = \Pol_\qq(\Pol_{\pp+1}(\vec\xx, \yy_1) \wdots \Pol_{\pp+1}(\vec\xx, \yy_\qq)),
\end{equation*}
where $\vec\xx$ stands for $\xx_1 \wdots \xx_\pp$; this is a formal identity, not using the RC-law or any specific relation; for instance, it says that $\Pol_3(\xx, \yy_1, \yy_2)$, that is, $(\xx\op\yy_1)\op(\xx\op\yy_2)$, is also $\Pol_2(\Pol_2(\xx, \yy_1), \Pol_2(\xx, \yy_2)$). With the same convention, one deduces 
\begin{equation}
\label{E:PolId2}
\SPol_{\pp+\qq}(\vec\xx, \yy_1 \wdots \yy_\qq) = \SPol_\pp(\vec\xx) \cdot \SPol_\qq(\Pol_{\pp+1}(\vec\xx, \yy_1 ) \wdots \Pol_{\pp+1}(\vec\xx, \yy_\qq)).
\end{equation}
Now, assume that $\gga, \hha$ lie in~$\NNNN^{(\SS)}$. Write $\gga = \ss_1 \pdots \ss_\pp$ and $\hha = \tt_1 \pdots\tt_\qq$ with $\ss_1 \wdots \tt_\qq$ in~$\SS$. The right equality in~\eqref{E:Nu} gives $\nu(\gga\hha) = \SPol_{\pp+\qq}(\ss_1 \wdots \ss_\pp, \tt_1 \wdots \tt_\qq)$. On the other hand, we have $\nu(\gga) = \SPol_\pp(\ss_1 \wdots \ss_\pp)$ and the left equality in~\eqref{E:Nu} implies $\psi(\gga)(\tt) = \Pol_{\pp+1}(\ss_1 \wdots \ss_\pp, \tt)$ for every~$\tt$, whence in particular 
$$\nu(\psi(\gga)(\tt)) = \SPol_\qq(\Pol_{\pp+1}(\ss_1 \wdots \ss_\pp, \tt_1) \wdots \Pol_{\pp+1}(\ss_1 \wdots \ss_\pp, \tt_\qq)).$$
Merging with~\eqref{E:PolId2}, we obtain the left formula in~\eqref{E:Product}.

Finally, assume $\ss \in \SS$. On the one hand, \eqref{E:IStrBis} gives $\nu(\gga\hha\ss) = \nu(\gga\hha) \psi(\gga\hha)(\ss)$. On the other hand, the left formula in~\eqref{E:Product} gives
\begin{multline*}
\nu(\gga\hha\ss) 
= \nu(\gga) \cdot \nu(\psi(\gga)[\hha\ss]) 
= \nu(\gga) \cdot \nu(\psi(\gga)[\hha] \cdot \psi(\gga)(\ss))\\
= \nu(\gga) \cdot \nu(\psi(\gga)[\hha]) \cdot \psi(\psi(\gga)[\hha])(\psi(\gga)(\ss)) 
= \nu(\gga\hha) \cdot \psi(\psi(\gga)[\hha])(\psi(\gga)(\ss)).
\end{multline*}
Merging the two expressions and using the assumption that $\MM$ is left-cancellative, we deduce $\psi(\gga\hha)(\ss) = \psi(\psi(\gga)[\hha])(\psi(\gga)(\ss))$, which is the right equality in~\eqref{E:Product}.
\end{proof}

We now connect left-divisibility with the $I$-structure, in the case of a right-$I$-structure of shape~$(\NNNN^{(\SS)}, \SS)$, that is, a genuine right-$I$-structure.

\begin{lemm}
\label{L:IStrDiv}
Assume that $\nu$ is a right-$I$-structure based on~$\SS$ in a monoid~$\MM$. 

\ITEM1 There exists a length function on~$\MM$ and $\SS$ is the atom set in~$\MM$.

\ITEM2 The map~$\nu$ is compatible with left-division in the sense that, for all~$\gga, \hha$ in~$\NNNN^{(\SS)}$, we have $\gga \dive \hha$ in~$\NNNN^{(\SS)}$ if and only if $\nu(\gga) \dive \nu(\hha)$ holds in~$\MM$. 

\ITEM3 The monoid~$\MM$ admits right-lcms.
\end{lemm}

\begin{proof}
\ITEM1 Defining $\lambda(\gg)$ to be the length of~$\nu\inv(\gg)$ provides a function from~$\MM$ to~$\NNNN$ that satisfies $\lambda(1) = 0$, $\lambda(\gg\hh) = \lambda(\gg) + \lambda(\hh)$, and $\lambda(\ss) = 1$ for every~$\ss$ in~$\SS$. It follows that $\MM$ contains no nontrivial invertible element, and that $\SS$ is the atom set of~$\MM$. 

\ITEM2 Assume $\gga \dive \hha$ in~$\NNNN^{(\SS)}$. For an induction on length, we may assume $\hha = \gga\ss$ with~$\ss$ in~$\SS$. Now, by~\eqref{E:IStrBis}, we have $\nu(\hha) = \nu(\gga) \psi(\gga)(\ss)$, whence $\nu(\gga) \dive \nu(\hha)$ in~$\MM$. Conversely, assume $\nu(\gga) \dive \nu(\hha)$. Again, it is enough to consider the case $\nu(\hha) = \nu(\gga)\ss$ with~$\ss$ in~$\SS$. Now, as $\psi(\gga)$ is bijective, there exists a unique~$\rr$ in~$\SS$ satisfying $\psi(\gga)(\rr) = \ss$, and, by~\eqref{E:IStrBis}, we have then $\nu(\gga\rr) = \nu(\gga)\psi(\gga)(\rr) = \nu(\gga)\ss = \nu(\hha)$, whence $\hha = \gga\rr$ since $\nu$ is injective, and $\gga \dive \hha$ in~$\NNNN^{(\SS)}$.

\ITEM3 The monoid~$\NNNN^{(\SS)}$ admits right-lcms, and \ITEM2 enables us to transfer the result to~$\MM$. So, let $\gg, \hh$ belong to~$\MM$. Put $\gga = \nu\inv(\gg)$ and $\hha = \nu\inv(\hh)$. Let $\gga\hha'$ be the right-lcm of~$\gga$ and~$\hha$ in~$\NNNN^{(\SS)}$. By~\ITEM2, $\nu(\gga\hha')$ is a common right-multiple of~$\gg$ and~$\hh$ in~$\MM$. Now, assume that $\ff$ is a common right-multiple of~$\gg$ and~$\hh$ in~$\MM$. By~\ITEM2 again, we have $\gga \dive \nu\inv(\ff)$ and~$\hha \dive \nu\inv(\ff)$ in~$\NNNN^{(\SS)}$, whence $\gga\hha' \dive \nu\inv(\ff)$. By~\ITEM2 once more, this implies $\nu(\gga\hha') \dive \ff$ in~$\MM$. So $\nu(\gga\hha')$ is a right-lcm of~$\gg$ and~$\hh$ in~$\MM$, and $\MM$ admits right-lcms.
\end{proof}

We can now easily complete the proof of Proposition~\ref{P:IType}.

\begin{proof}[Proof of Proposition~\ref{P:IType}]
Assume that $\nu$ is a right-$I$-structure on~$\MM$, bas\-ed on a set~$\SS$. By Lemma~\ref{L:IStrDiv}\ITEM1, $\SS$ must be the atom set of~$\MM$, and the assumption that $\MM$ is finitely generated implies that $\SS$ is finite. 

Next, the free Abelian monoid~$\NNNN^{(\SS)}$ satisfies the assumptions of Lemma~\ref{L:NuRC} and, therefore, the latter applies. Hence, if we define $\ss \op \tt = \psi(\ss)(\tt)$, then $(\SS, \op)$ is an RC-quasigroup. Moreover, as $\SS$ is finite, $(\SS, \op)$ is bijective by~\cite[Theorem~2]{RumYB}.

Assume $\ss \not= \tt \in \SS$. By~\eqref{E:IStrBis}, we have $\ss(\ss \op \tt) = \nu(\ss\tt) = \nu(\tt\ss) = \tt(\tt \op \ss)$, whereas, by Lemma~\ref{L:IStrDiv}\ITEM1 and \ITEM3, the monoid~$\MM$ admits unique right-lcms. Moreover, as $\SS$ is finite, the argument of~\cite[Lemma~8.1.2(6)]{JeOBook} shows that $\MM$ must be left-cancellative (for that point we have no alternative method and it is useless to repeat the original argument, which easily adapts). It follows that  $\ss(\ss \op \tt)$ is necessarily the right-lcm of~$\ss$ and~$\tt$, and we must have $\ss\op\tt = \ss\under\tt$. 

Finally, as $\MM$ is left-cancellative, admits a length function, and admits right-lcms, and as $\SS$ is the atom set of~$\MM$, it follows from~\cite[Prop.~4.1]{Dfx} that the list of all relations $\ss(\ss\under\tt) = \tt(\tt\under\ss)$ with $\ss \not= \tt \in \SS$ make a presentation of~$\MM$. By definition of~$\op$, this means that $\MM$ is the structure monoid of~$(\SS, \op)$. 

Finally, the connection between~$\nu$ and the polynomials~$\SPol$ is given by the right formula in~\eqref{E:Nu}. The uniqueness of~$\nu$ follows, as $\SS$ is the atom set of~$\MM$, and $\op$ is the only possible extension of the right-complement operation outside the diagonal that admits bijective left-translations, so they only depend on~$\MM$.
\end{proof}

As observed in~\cite{JeO}, \eqref{E:IStrBis} is reminiscent of a semi-direct product. For further reference, we now describe the connection formally, here in the slightly extended context of Lemma~\ref{L:NuRC}; the point is to have the formulas of~\eqref{E:Product} at hand.

\begin{prop}
\label{P:Wreath}
Assume that $\MM$ is a left-cancellative monoid, $\NN$ is either~$\NNNN$ or $\Zd$ for some~$\dd$, and $\nu$ maps~$\NN^{(\SS)}$ to~$\MM$. Then the following are equivalent:

\ITEM1 The map~$\nu$ is a right-$I$-structure of shape~$(\NN^{(\SS)}, \SS)$ on~$\MM$;

\ITEM2 There exists a map $\psi : \NN^{(\SS)} \to \Sym_\SS$ such that $\gg \mapsto (\nu\inv(\gg), \psi(\nu\inv(\gg))\inv)$ defines an injective homomorphism of~$\MM$ to the wreath product $\NN \wr \Sym_\SS$ whose first component is a bijection.
\end{prop}

\begin{proof}
Assume that $\nu$ is a right $I$-structure of shape~$\NN^{(\SS)}$ on~$\MM$. Let $\psi$ be the map from~$\NN^{(\SS)}$ to~$\Sym_\SS$ of~\eqref{E:IStrBis}. Define $\iota: \MM \to \NN \wr \Sym_\SS$ by $\iota(\gg) = (\nu\inv(\gg), \psi(\nu\inv(\gg))\inv)$, where $\NN \wr \Sym_\SS$ is  $\NN^{(\SS)} \rtimes \Sym_\SS$ with $\Sym_\SS$ acting on~$\NN^{(\SS)}$ by permuting positions.

First, $\iota$ is injective and its first component is bijective, since $\nu\inv$ is bijective. 

In order to prove that \ITEM1 implies~\ITEM2, the point is to check that $\iota$ is a homomorphism. Let $\gg, \hh$ belong to~$\MM$. Putting $\gga = \nu\inv(\gg)$, $\hha = \nu\inv(\hh)$, $\sigma = \psi(\gga)$, and $\tau = \psi(\hha)$, we find $\iota(\gg) = (\gga, \sigma\inv)$ and $\iota(\hh) = (\hha, \tau\inv)$, whence, in $\NN \wr \Sym_\SS$,
\begin{equation}
\label{E:YBSemi}
\iota(\gg)\iota(\hh) = (\gga \sigma\inv[\hha], \sigma\inv \comp \tau\inv).
\end{equation}
On the other hand, the monoid~$\NN^{(\SS)}$ satisfies~\eqref{E:Condition}, so Lemma~\ref{E:NuRCBis} applies and, using the left formula in~\eqref{E:Product}, we find 
$$\nu(\gga \sigma\inv[\hha]) = \nu(\gga) \nu(\psi(\gga)(\sigma\inv[\hha])) = \gg \nu(\sigma[\sigma\inv[\hha]]) = \gg \nu(\hha) = \gg\hh,$$
whence $\nu\inv(\gg\hh) = \gga \sigma\inv[\hha]$. Using the right formula in~\eqref{E:Product} similarly, we find
$$\psi(\gga \sigma\inv[\hha]) = \psi(\psi(\gga)[\sigma\inv[\hha]]) \comp \psi(\gga) = \psi(\sigma[\sigma\inv[\hha]]) \comp \sigma = \psi[\hha] \comp \sigma = \tau \comp \sigma.$$
We deduce $\iota(\gg\hh) = (\gga \sigma\inv[\hha], \sigma\inv \comp \tau\inv)$. So, by~\eqref{E:YBSemi}, $\iota$ is a homomorphism.

Conversely, assume~\ITEM2 with $\nu$ and~$\psi$ ensuring that $\iota$ is an embedding. Then the above computation gives, for all $\gga, \hha$ in~$\NN^{(\SS)}$, the equality $\nu(\gga\hha) = \nu(\gga)\nu(\psi(\gga)[\hha])$, whence in particular $\nu(\gga\ss) = \nu(\gga)\psi(\gga)(\ss)$ for~$\gga$ in~$\NN^{(\SS)}$ and~$\ss$ in~$\SS$. So $\nu$ is a right-$I$-structure of shape~$(\NN^{(\SS)}, \SS)$ on~$\MM$, and \ITEM2 implies~\ITEM1
\end{proof}

\section{Coxeter-like groups}
\label{S:Coxeter}

We now use the RC-calculus of Section~\ref{S:RCCalculus} and the $I$-structure of Section~\ref{S:IStructure} and to solve what can be called the quest of a Coxeter group, namely constructing for every group associated with a finite RC-quasigroup a finite quotient that plays the role played by Coxeter groups in the case of spherical Artin--Tits groups. This finite quotient is not the group~$G_X^0$ of~\cite{Eti} in general.

In the case of Artin's braid group~$\BB_\nn$, the seminal example of a Garside group, the Garside structure $(\BP\nn, \Delta_\nn)$ is connected with the symmetric group~$\Sym_\nn$. Precisely, the group~$\BB_\nn$ and the monoid~$\BP\nn$ admit the (Artin) presentation
\begin{equation}
\label{E:BraidPres}
\bigg\langle \sig1, ..., \sig{n-1} \ \bigg\vert\ 
\begin{matrix}
\sig\ii \sig\jj = \sig\jj \sig\ii 
&\text{for} &\vert\ii-\jj\vert \ge 2\\
\sig\ii \sig \jj \sig\ii = \sig \jj \sig\ii \sig \jj 
&\text{for} &\vert\ii-\jj\vert = 1
\end{matrix}
\ \bigg\rangle,
\end{equation}
and $\Sym_\nn$ is the quotient of~$\BB_\nn$ obtained by adding the relations~$\sigg\ii2=1$ to~\eqref{E:BraidPres}. Then there exists a set-theoretic section~$\sigma : \Sym_\nn \to \BB_\nn$ whose image is the family of all divisors of~$\Delta_\nn$ in~$\BP\nn$, and a presentation both of the group~$\BB_\nn$ and the monoid~$\BP\nn$ in terms of the image of~$\sigma$ consists of all relations $\sigma(\ff)\sigma(\gg) = \sigma(\hh)$ with $\ff, \gg, \hh$ in~$\Sym_\nn$ satisfying $\LGG{}\ff + \LGG{}\gg = \LGG{}\hh$, where $\LGG{}\ff$ is the length of~$\ff$ (minimal number of adjacent transpositions in a decomposition of~$\ff$). Thus $\BB_\nn$ appears as a sort of unfolded version of~$\Sym_\nn$ where permutation length is used to get rid of torsion.

This is the situation we shall obtain in our current context. To make things precise, we introduce a notion of a Garside germ~\cite{Dif}. If a set~$\SS$ positively generates a group~$\GG$ (that is, every element of~$\GG$ can be expressed as a product of elements of~$\SS$), we denote by~$\LGG\SS\gg$ the length of a shortest $\SS$-decomposition of~$\gg$.

\begin{defi}
\label{D:Germ}
If $(\MM, \Delta)$ is a Garside monoid and $\GG$ is its group of fractions, a surjective homomorphism $\pi : \GG \to \GGb$ is said to \emph{provide a Garside germ for~$(\GG, \MM, \Delta)$} if there exists a map $\sigma: \GGb \to \MM$ such that $\pi \comp \sigma$ is the identity, the image of~$\sigma$ is the family of all divisors of~$\Delta$ in~$\MM$, and $\MM$ admits the presentation
\begin{equation}
\label{GermPres}
\langle\ \sigma(\GGb) \mid \{ \sigma(\ff)\sigma(\gg) = \sigma(\ff\gg) \mid \ff, \gg \in \GGb \text{ and } \LGG{\SSb}\ff + \LGG{\SSb}\gg\ = \LGG{\SSb}{\ff\gg}\}\ \rangle,
\end{equation}
where $\SSb$ is the image under~$\pi$ of the atom set of~$\MM$.
\end{defi}

In the context of Definition~\ref{D:Germ}, every element of~$\GG$ can be written as $\Delta^\pp \gg$ for some $\pp$ in~$\ZZZZ$ and some~$\gg$ in~$\MM$, implying that $\SSb$ positively generates~$\GGb$ and making $\LGG\SSb\gg$ meaningful. The term \emph{germ} stems from~\cite{DiM, Dif} where the germ derived from~$(\GGb, \SSb)$ is defined to be~$(\GGb, \OP)$ where $\OP$ is the partial binary operation so that $\ff \OP \gg = \hh$ holds for $\ff \gg = \hh$ with $\LGG{\SSb}\ff + \LGG{\SSb}\gg = \LGG{\SSb}\hh$. The monoid and the group defined by~\eqref{GermPres} are then said to be generated by the germ~$(\GGb, \OP)$. The situation described in Definition~\ref{D:Germ} corresponds to $(\GGb, \OP)$ being a germ generating~$\GG$. When it is so, the maps~$\pi$ and~$\sigma$ induce mutually inverse isomorphisms between the finite lattice made by the divisors of~$\Delta$ in~$\MM$ and $(\GGb, \leSb)$ where $\ff \leSb \gg$ means $\LGG{\SSb}\ff + \LGG{\SSb}{\ff\inv \gg} = \LGG{\SSb}\gg$. The Hasse diagram of these partial orders then coincides with the Cayley graph of the germ~$(\GGb, \OP)$ with respect to the generating set~$\SSb$. 

The above mentioned results for the braid group~$\BB_\nn$ and the symmetric group~$\Sym_\nn$ can be summarized into the statement that collapsing~$\sigg\ii2$ to~$1$ for every~$\ii$ provides a Garside germ for~$(\BB_\nn, \BP\nn, \Delta_\nn)$, with associated quotient~$\Sym_\nn$. It is known---see \cite{Bou} or \cite[Chapter~IX]{Garside}---that similar results hold for every Artin--Tits group of spherical type: if $(\WW, \SS)$ is a spherical Coxeter system (that is, $\WW$ and $\SS$ are finite), and $\GG$ and~$\MM$ are the associated Artin--Tits group and monoid, and $\Delta$ is the right-lcm of atoms in~$\MM$, then collapsing~$\ss^2$ to~$1$ for every~$\ss$ in~$\SS$ provides a Garside germ for~$(\GG, \MM, \Delta)$, with associated quotient~$\WW$. 

It is then natural to ask whether similar results hold for every Garside group, namely whether some finite quotient provides a Garside germ, that is, in other words, whether there exists an associated Coxeter-like group enjoying all the nice properties known in the Artin--Tits case. No general answer is known, but we shall establish a complete positive answer for structure groups of finite RC-quasigroups. Indeed, we define for every finite RC-quasigroup a notion of  \emph{class} and prove:

\begin{prop}
\label{P:Cox}
Assume that $(\SS, \op)$ is an RC-quasigroup of cardinal~$\nn$ and class~$\dd$. Let $\GG, \MM$ be the associated group and monoid, and $\Delta$ be the right-lcm of~$\SS$ in~$\MM$. Then collapsing $\ss^{[\dd]}$ to~$1$ in~$\GG$ for every~$\ss$ in~$\SS$, where $\ss^{[\dd]}$ stands for~$\SPol_\dd(\ss \wdots \ss)$, gives a finite group~$\GGb$ that provides a Garside germ for~$(\GG, \MM, \Delta^{\dd-1})$. The group~$\GGb$ has $\dd^\nn$ elements and the kernel of the projection is (isomorphic to)~$\ZZZZ^\nn$.
\end{prop}

In the framework of Proposition~\ref{P:Cox}, the finite group~$\GGb$ will be called the \emph{Coxeter-like} group associated with~$(\SS, \op)$ and~$\dd$.

The proof, which is not difficult, consists in using the $I$-structure to carry the results from the (trivial) case of~$\ZZZZ^\nn$ to the case of an arbitrary group of $I$-type. It will be decomposed into several easy steps. First we define the class. 

\begin{defi}
\label{D:Class}
An RC-quasigroup $(\SS, \op)$ is said to be \emph{of class~$\dd$} it it satisfies
\begin{equation}
\forall\ss, \tt \in \SS\ (\ \Pol_{\dd+1}(\ss \wdots \ss, \tt) = \tt\ ).
\end{equation}
\end{defi}

So an RC-quasigroup is of class~$1$ if $\ss \op \tt = \tt$ holds for all~$\ss, \tt$, and it is of class~$2$ if $(\ss \op \ss) \op (\ss \op \tt) = \tt$ holds for all~$\ss, \tt$. 

\begin{lemm}
\label{L:Class}
Every RC-quasigroup of cardinal~$\nn$ is of class~$\dd$ for some~$\dd < (\nn^2)!$.
\end{lemm}

\begin{proof}
Let $(\SS, \op)$ be a finite RC-quasigroup with cardinal~$\nn$. By Rump's theorem, $(\SS, \op)$ must be bijective, that is, the map $\Psi : (\ss, \tt) \mapsto (\ss \op \tt, \tt \op \ss)$ is bijective on~$\SS \times \SS$. Consider $\Phi : (\ss, \tt) \mapsto (\ss \op \ss, \ss \op \tt)$ on~$\SS^2$. Assume $(\ss, \tt) \not= (\ss', \tt')$. For $\ss \not = \ss'$, we have $\Psi(\ss, \ss) \not= \Psi(\ss', \ss')$, hence $\ss \op \ss \not= \ss' \op \ss'$, and $\Phi(\ss, \tt) \not= \Phi(\ss', \tt')$. For $\ss = \ss'$, we must have $\tt \not= \tt'$, whence $\ss \op \tt \not= \ss \op \tt'$ and, again, $\Phi(\ss, \tt) \not= \Phi(\ss', \tt')$ since left-translations of~$\op$ are injective. So $\Phi$ is injective, hence bijective on the finite set~$\SS \times \SS$. As $\SS \times \SS$ has cardinal~$\nn^2$, the order of~$\Phi$ in~$\Sym_{\SS \times \SS}$ is at most~$(\nn^2)!$. So there exists $\dd < (\nn^2)!$ such that $\Phi^{\dd+1}$ is the identity. Now, an easy induction gives $\Phi^\mm(\ss, \tt) = (\Pol_\mm(\ss \wdots \ss, \ss), \Pol_\mm(\ss \wdots \ss, \tt))$ for every~$\mm$. So $\Phi^{\dd+1} = \mathrm{id}$ implies $\Pol_{\dd+1}(\ss \wdots \ss, \tt) = \tt$ for all~$\ss, \tt$ in~$\SS$.
\end{proof}

There exist finite RC-quasigroups with an arbitrarily high minimal class. Indeed, let $\SS = \{\tta_1 \wdots \tta_\nn\}$ and $\ss \op \tt = \ff(\tt)$ where $\ff$ maps $\tta_\ii$ to~$\tta_{\ii+1 (\mathrm{mod}\,\nn)}$ for every~$\ii$. Then, for all~$\pp, \ii$ and $\ss_1 \wdots \ss_\pp$ in~$\SS$, we have $\Pol_{\pp+1}(\ss_1 \wdots \ss_\pp, \tta_\ii) = \tta_{\ii+\pp\, (\mathrm{mod}\,\nn)}$. Hence $(\SS, \op)$ is of class~$\dd$ if and only if $\dd$ is a multiple of~$\nn$, and the minimal class is~$\nn$.

We shall establish Proposition~\ref{P:Cox} using the $I$-structure on the monoid~$\MM$, which exists by Proposition~\ref{P:RCItype}. As in Section~\ref{S:IStructure}, the $I$-structure will be denoted by~$\nu$, and the associated map from~$\NNNN^\SS$ to~$\Sym_\SS$ as defined in~\eqref{E:IStrBis} will be denoted by~$\psi$. As mentioned at the end of Section~\ref{S:IStructure}, $\nu$ and $\psi$ respectively extend into a bijection from~$\ZZZZ^\SS$ to~$\GG$ and a map from~$\ZZZZ^\SS$ to~$\Sym_\SS$ that still satisfy~\eqref{E:IStr} and~\eqref{E:IStrBis}.

\begin{lemm}
\label{L:Frozen}
Assume that $(\SS, \op)$ is an RC-quasigroup of class~$\dd$ and $\MM$ is the associated monoid. For~$\ss$ in~$\SS$ and $\qq \ge 0$, let $\ss^{[\qq]} = \SPol_\qq(\ss \wdots \ss)$. Then $\nu(\ss^\dd \gga) = \ss^{[\dd]} \nu(\gga)$ holds for all $\ss$ in~$\SS$ and $\gga$ in~$\NNNN^{(\SS)}$. The permutation~$\psi(\ss^\dd)$ is the identity and, for all~$\ss, \tt$ in~$\SS$, the elements $\ss^{[\dd]}$ and $\tt^{[\dd]}$ commute in~$\MM$.
\end{lemm}

\begin{proof}
Assume $\aa = \tt_1 \pdots \tt_\qq$ with $\tt_1 \wdots \tt_\qq$ in~$\SS$. Proposition~\ref{P:RCItype} implies 
\begin{align*}
\nu(\ss^\dd \gga) 
&= \SPol_{\dd+q}(\ss \wdots \ss, \tt_1 \wdots \tt_\qq)\\
&= \SPol_\dd(\ss \wdots \ss) \SPol_\qq(\Pol_{\dd+1}(\ss \wdots \ss, \tt_1) \wdots \Pol_{\dd+1}(\ss, \wdots \ss, \tt_\qq))\\
&= \SPol_\dd(\ss \wdots \ss) \SPol_\qq(\tt_1 \wdots \tt_\qq) = \nu(\ss^\dd) \nu(\tt_1 \pdots \tt_\qq) = \ss^{[\dd]} \nu(\gga),
\end{align*}
in which the second equality comes from expanding the terms and the third one from the assumption that $\MM$ is of class~$\dd$. Applying with $\gga = \tt$ in~$\SS$ and merging with $\nu(\ss^\dd \tt) = \nu(\ss^\dd) \, \psi(\ss^\dd)(\tt)$, we deduce that $\psi(\ss^\dd)$ is the identity. On the other hand, applying with $\gga = \tt^{[\dd]}$, we find $\ss^{[\dd]}\tt^{[\dd]} = \nu(\ss^\dd \tt^\dd) = \nu(\tt^\dd \ss^\dd) = \tt^{[\dd]} \ss^{[\dd]}$.
\end{proof}

\begin{lemm}
\label{L:Delta}
\ITEM1 Assume that $(\SS, \op)$ is a finite RC-quasigroup, $\MM$ is the associated monoid, and $\dd \ge 2$ holds. Let $\Deltaz = \prod_{\ss \in \SS}\ss$ in~$\NNNN^\SS$ and $\Delta_\dd = \nu(\Deltaz^{\dd-1})$. Then we have $\Delta_\dd = \Delta^{\dd-1}$ where $\Delta$ is the right-lcm of~$\SS$, and $\Delta_\dd$ is a Garside element in~$\MM$.

\ITEM2 If, moreover, $(\SS, \op)$ is of class~$\dd$, then $\Delta^\dd$ and $(\Delta_\dd)^\dd$ lie in the centre of~$\MM$.
\end{lemm}

\begin{proof}
\ITEM1 By Lemma~\ref{L:Lcm}, we have $\Delta = \SPol_\nn(\ss_1 \wdots \ss_\nn) = \nu(\Deltaz)$, where $(\ss_1 \wdots \ss_\nn)$ is any enumeration of~$\SS$. In other words, we have $\Delta = \Delta_2$. Now, we observe that $\ff[\Deltaz] = \Deltaz$ holds in~$\NNNN^\SS$ for every~$\ff$ in~$\Sym_\SS$ since every element of~$\SS$ occurs once in the definition of~$\Deltaz$. By~\eqref{E:Product}, we deduce 
\begin{equation}
\label{E:Delta}
\nu(\gga\Deltaz) = \nu(\gga) \nu(\psi(\gga)[\Deltaz]) = \nu(\gga) \nu(\Deltaz),
\end{equation}
whence $\nu(\Deltaz^\kk) = \nu(\Deltaz)^\kk$ for every~$\kk$ and, in particular, $\Delta_\dd = \nu(\Deltaz)^{\dd-1} = \Delta^{\dd-1}$. By Proposition~\ref{P:GarMon}, $\Delta$ is a Garside element in~$\MM$. It is standard that this implies that every power of~$\Delta$ is also a Garside element, hence, in particular, so is~$\Delta_\dd$.

\ITEM2 Assume now that $(\SS, \op)$ is of class~$\dd$. Let $\tt$ belong to~$\SS$. Then, by~\eqref{E:Delta}, we obtain $\nu(\tt \Deltaz^\dd) = \nu(\tt) \nu(\Deltaz^\dd) = \tt \Delta^\dd$. On the other hand, \eqref{E:Delta} and Lemma~\eqref{L:Frozen} give
\begin{equation}
\label{E:Delta1}
\Delta^\dd = \nu(\Deltaz^\dd) = \prod_{\ss \in \SS}\ss^{[\dd]} \text{\ and \ }\nu(\Deltaz^\dd \tt) = \prod_{\ss \in \SS}\ss^{[\dd]} \tt = \Delta^\dd \tt.
\end{equation}
Merging the values of~$\nu(\tt\Deltaz^\dd)$ and $\nu(\Deltaz^\dd \tt)$, we obtain $\tt \Delta^\dd = \Delta^\dd \tt$, so that $\Delta^\dd$, hence its power~$(\Delta_\dd)^\dd$ as well, lies in the centre of~$\MM$. 
\end{proof}

We can now introduce the equivalence relation on~$\ZZZZ^\SS$ that, when carried to~$\GG$, induces the expected quotient of~$\GG$ (and~$\MM)$. For~$\gga$ in~$\ZZZZ^\SS$, we denote by~$\Nb\ss\gga$ the (well-defined) algebraic number of letters~$\ss$ in any $\SS$-decomposition of~$\gga$.

\begin{lemm}
\label{L:Congruence}
Assume that $(\SS, \op)$ is an RC-quasigroup of class~$\dd$ and $\MM$ and $\GG$ are the associated monoid and group.  For~$\gga, \gga'$ in~$\ZZZZ^\SS$, write $\gga \equivz \gga'$ if $\Nb\ss\gga = \Nb\ss{\gga'} \pmod\dd$ holds for every~$\ss$ in~$\SS$.

\ITEM1 For~$\gg, \gg'$ in~$\MM$, declare $\gg \equiv \gg'$ for $\nu\inv(\gg) \equivz \nu\inv(\gg')$. Then $\equiv$ is an equivalence relation on~$\MM$ that is compatible with left- and right-multiplication. The class of~$1$ is the Abelian submonoid~$\MM_1$ of~$\MM$ generated by the elements~$\ss^{[\dd]}$ with~$\ss$ in~$\SS$.

\ITEM2 For~$\gg, \gg'$ in~$\GG$, declare that $\gg \equiv \gg'$ holds if there exist~$\hh, \hh'$ in~$\MM$ and $\rr, \rr'$ in~$\ZZZZ$ satisfying $\gg = \Delta^{\dd\rr} \hh$, $\gg' = \Delta^{\dd\rr'} \hh'$, and $\hh \equiv \hh'$. Then $\equiv$ is a congruence on~$\GG$, and the kernel of the projection of~$\GG$ to~$\GG{/}{\equiv}$ is the group of fractions of~$\MM_1$.
\end{lemm}

\begin{proof}
\ITEM1 As $\nu$ is bijective, carrying the equivalence relation~$\equivz$ of~$\NNNN^\SS$ to~$\MM$ yields an equivalence relation on~$\MM$. Assume $\gg \equiv \gg'$. Let $\gga = \nu\inv(\gg)$ and $\gga' =\nobreak \nu\inv(\gg')$. Without loss of generality, we may assume $\gga' = \gga \ss^\dd = \ss^\dd \gga$ for some $\ss$ in~$\SS$. Applying~\eqref{E:Product} and Lemma~\ref{L:Frozen}, we obtain $\psi(\gga') = \psi(\psi(\ss^\dd)[\gga] \comp \psi(\ss^\dd) = \psi(\gga)$. Let~$\tt$ belong to~$\SS$. Using~\eqref{E:Product} again, we deduce
\begin{multline*}
\HS{5}\gg \cdot \psi(\gga)(\tt) = \nu(\gga) \cdot \psi(\gga)(\tt) = \nu(\gga\tt) \\ \equiv \nu(\gga'\tt) = \nu(\gga') \cdot \psi(\gga')(\tt) = \nu(\gga') \cdot\psi(\gga)(\tt) = \gg' \cdot \psi(\gga)(\tt).\HS{5}
\end{multline*}
As $\psi(\gga)(\tt)$ takes every value in~$\SS$ when $\tt$ ranges over~$\SS$, we deduce that $\equiv$ is compatible with right-multiplication. 
On the other hand, $\gga \equivz \gga'$ implies $\ff[\gga] \equivz\nobreak \ff[\gga']$ for every permutation~$\ff$ of~$\SS$. Let~$\tt$ belong to~$\SS$. Always by~\eqref{E:Product}, we obtain 
$$\tt \cdot \gg = \tt \cdot \nu(\gga) = \nu(\tt \cdot \psi(\tt)^{-1}[\gga]) \equiv \nu(\tt \cdot \psi(\tt)^{-1}[\gga']) = \tt \cdot \nu(\gga') = \tt \cdot \gg',$$
and $\equiv$ is compatible with left-multiplication by~$\SS$. 

The $\equivz$-class of~$1$ in~$\NNNN^\SS$ is the free Abelian submonoid generated by the elements~$\ss^\dd$ with~$\ss$ in~$\SS$. The $\equiv$-class of~$1$ in~$\MM$ consists of the $\nu$-image of the products of such elements~$\ss^\dd$. By Lemma~\ref{L:Frozen}, the latter are the products of elements~$\ss^{[\dd]}$.

\ITEM2 First, $\equiv$ is well-defined. As $\Delta^\dd$ is a Garside element in~$\MM$, every element of~$G$ admits an expression $\Delta^{\dd\rr} \hh$ with $\rr$ in~$\ZZZZ$ and $\hh$ in~$\MM$. This expression is not unique, but, if we have $\gg = \Delta^{\dd\rr} \hh = \Delta^{\dd\rr_1} \hh_1$ with, say, $\rr_1 < \rr$, then, as $\MM$ is left-cancellative, we must have $\hh_1 = \Delta^{\dd(\rr-\rr_1)} \hh$, whence $\hh_1 \equiv \hh$ by~\eqref{E:Delta1}. So, for every~$\hh'$ in~$\MM$, the relations $\hh \equiv \hh'$ and $\hh_1 \equiv \hh'$ are equivalent.

That $\equiv$ is a equivalence relation is easy. Its compatibility with multiplication on~$\GG$ follows from the compatibility on~$\MM$ and the fact that $\Delta^\dd$ is central in~$\GG$. 

Finally, the $\equiv$-class of~$1$ in~$\GG$ consists of all elements~$\Delta^{\dd\rr} \hh$ with~$\hh$ in~$\MM_1$. As $\Delta^\dd$ belongs to~$\MM_1$, this is the group of fractions of~$\MM_1$ in~$\GG$, hence the free Abelian subgroup of~$\GG$ generated by the elements~$\ss^{[\dd]}$ with $\ss$ in~$\SS$.
\end{proof}

\begin{proof}[Proof of Proposition~\ref{P:Cox}]
Let $\GGb$ be the quotient-group~$\GG{/}{\equiv}$. By Lemma~\ref{L:Congruence}, the kernel of the projection of~$\GG$ onto~$\GGb$ is a free Abelian group of rank~$\nn$, hence it is isomorphic to~$\ZZZZ^\SS$. The cardinality of~$\GGb$ is the number of $\equiv$-classes in~$\GG$. As every element of~$\GG$ is $\equiv$-equivalent to an element of~$\MM$, this number is also the number of $\equiv$-classes in~$\MM$, hence the number~$\dd^\nn$ of $\equivz$-classes in~$\NNNN^\SS$, and we have $\GGb = \MM{/}{\equiv}$.

By definition, $\ss^{[\dd]} \equiv 1$ holds for~$\ss$ in~$\SS$. Conversely, the congruence~$\equivz$ on~$\ZZZZ^\SS$ is generated by the pairs $(\ss^\dd, 1)$ with $\ss$ in~$\SS$, hence the congruence~$\equiv$ on~$\GG$ is generated by the pairs $(\ss^{[\dd]}, 1)$ with~$\ss$ in~$\SS$. Hence a presentation of~$\GGb$ is obtained by adding to the presentation~\eqref{E:Str} of~$\GG$ and of~$\MM$ the $\nn$~relations $\ss^{[\dd]} = 1$ with~$\ss$ in~$\SS$. 

By construction, the bijection~$\nu$ is compatible with the congruences~$\equivz$ on~$\ZZZZ^\SS$ and~$\equiv$ on~$\GG$, so it induces a bijection~$\nub$ of~$\NNNN^\SS{/}{\equivz}$, which is $(\Zd)^\SS$, onto~$\MM{/}{\equiv}$, which is~$\GGb$, providing a commutative diagram
\vspace{-8mm}
\begin{equation}
\label{E:Projection}
\VR(6,8)\begin{picture}(22,18)(0,5)
\rput[bc](0,0){\rnode{00}{$(\Zd)^\SS$}}
\rput[bc](25,0){\rnode{10}{$\GGb$\rlap{\ .}}}
\rput[c](0,12){\rnode{01}{$\NNNN^\SS$}}
\rput[c](25,12){\rnode{11}{$\phantom{(}\MM\phantom{)}$}}
\ncline[nodesep=2pt,linewidth=0.5pt]{->}{01}{11}\naput{$\nu$}
\ncline[nodesep=2pt,linewidth=0.5pt]{->}{01}{00}\tlput{$\piz$}
\ncline[nodesep=2pt,linewidth=0.5pt]{->}{11}{10}\trput{$\pi$}
\ncline[nodesep=2pt,linewidth=0.5pt]{->}{00}{10}\nbput{$\nub$}
\end{picture}
\end{equation}
Now, let $\sigmaz$ be the section of~$\piz$ from $(\Zd)^\SS$ to~$\NNNN^\SS$ that maps every $\equivz$-class to the unique $\nn$-tuple of $\{0 \wdots \dd-1\}^\SS$ that lies in that class, and let $\sigma: \GGb \to \MM$ be defined by $\sigma(\gg) = \nu(\sigmaz(\nub\inv(\gg))$. Then, for every~$\gg$ in~$\GGb$, we obtain
$$\pi(\sigma(\gg)) = \pi(\nu(\sigmaz(\nub\inv(\gg)) = \nub(\piz(\sigmaz(\nub\inv(\gg)) = \gg$$
since $\sigmaz$ is a section of~$\piz$. Hence $\sigma$ is a section of~$\pi$. Next, by construction, the image of~$\GGb$ under~$\sigma$ is the image under~$\nu$ of $\{0 \wdots \dd-1\}^\SS$, hence the image under~$\nu$ of the family of all left-divisors of~$\Deltaz^{\dd-1}$ in~$\NNNN^\SS$, hence the family of all left-divisors of~$\Delta^{\dd-1}$, that is, of~$\Delta_\dd$, in~$\MM$.

Finally, the relation $\sigma(\ff)\sigma(\gg) = \sigma(\ff\gg)$ holds in~$\MM$ if and only if the relation $\sigmaz(\nub(\ff)) \sigmaz(\nub(\gg)) = \sigmaz(\nub(\ff\gg))$ holds in~$\NNNN^\SS$, hence if and only if, for every~$\ii$, the sum of the $\ii$th coordinates of~$\nub(\ff)$ and~$\nub(\gg)$ does not exceed~$\dd-1$. This happens if and only if $\LGG\SS{\nub(\ff)} + \LGG\SS{\nub(\gg)} = \LGG\SS{\nub(\ff\gg)}$ holds in~$(\Zd)^\SS$, hence if and only if $\LGG\SSb{\ff} + \LGG\SSb{\gg} = \LGG\SSb{\ff\gg}$ holds in~$\GGb$. By construction, the family~$\SS$ is included in the image of~$\sigma$, and all length two relations of~\eqref{E:Str} belong to the previous list of relations, hence the latter make a presentation of~$\MM$. This completes the proof.
\end{proof}

\begin{exam}
\label{X:Cox}
For an RC-quasigroup of class~1, that is, satisfying $\ss \op \tt = \tt$ for all~$\ss, \tt$, the group~$\GG$ is a free Abelian group, $\GGb$ is trivial, and Proposition~\ref{P:Cox} here reduces to the isomorphism $\ZZZZ^\SS \cong \GG$.

For class~2, that is, when $(\ss \op \ss) \op (\ss \op \tt) = \tt$ holds for all~$\ss, \tt$, the element~$\Delta_\dd$ is the right-lcm of~$\SS$, it has $2^\nn$~divisors which are the right-lcms of subsets of~$\SS$, and the group~$\GGb$ is the order~$2^\nn$ quotient of~$\GG$ obtained by adding the relations $\ss(\ss \op \ss) = 1$. For instance, in the case of $\{\tta, \ttb\}$ with $\ss \op \tt = \ff(\tt)$, $\ff : \tta \mapsto \ttb \mapsto \tta$, the group~$\GG$ has the presentation $\PRES{\tta, \ttb}{\tta^2 = \ttb^2}$, the relations $\tta^{[2]} = \ttb^{[2]} = 1$ both amount to $\tta\ttb = 1$, and the associated Coxeter-like group~$\GGb$ is a cyclic group of order~$4$.

For class~3, let us consider as in Example~\ref{X:Lattice} the RC-quasigroup $\{\tta, \ttb, \ttc\}$ with $\ss \op \tt =\nobreak \ff(\tt)$ and $\ff : \tta {\mapsto} \ttb {\mapsto} \ttc {\mapsto} \tta$. The presentation of the associated group~$\GG$ is $\PRES{\tta, \ttb, \ttc}{\tta\ttc = \ttb^2, \ttb\tta = \ttc^2, \ttc\ttb = \tta^2}$. With the same notation as above, the smallest Garside element~$\Delta$ is~$\tta^3$. As the class of~$(\SS, \op)$ is~$3$, we consider here $\Delta_3 = \Delta^2 = \nobreak \tta^6$. The lattice $\Div(\Delta_3)$ has 27~elements, its Hasse diagram is the cube shown in Figure~\ref{F:Cube}. The latter is also the Cayley graph of the germ derived from~$(\GGb, \{\tta, \ttb, \ttc\})$, that is, the restriction of the Cayley graph of~$\GGb$ to the partial product of the germ. Adding to the above presentation the relations $\ss^{[3]} = 1$, that is, $\ss (\ss \op \ss )((\ss \op \ss )\op(\ss \op \ss )) = 1$, namely $\tta\ttb\ttc = \ttb\ttc\tta = \ttc\tta\ttb = 1$, here reducing to $\tta\ttb\ttc = 1$, yields for~$\GGb$ the presentation $\PRES{\tta, \ttb, \ttc}{\tta\ttc =\nobreak \ttb^2, \ttb\tta = \ttc^2, \ttc\ttb = \tta^2, \tta\ttb\ttc = 1}$. One can check that other presentations of~$\GGb$ are $\PRES{\tta, \ttb}{\tta = \ttb^2\tta\ttb, \ttb = \tta\ttb\tta^2}$ and $\PRES{\tta, \ttb}{\tta = \ttb^2\tta\ttb, \tta^3 = \ttb^3}$.
\end{exam}

\begin{figure}[htb]
$$\begin{picture}(65,53)(0,3)
\psset{unit=1.2mm}
\psset{linewidth=0.8pt}
\psset{arrowsize=3pt}
\psset{dotsep=1.5pt}
\psset{dash=2pt 1pt}
\psset{doublesep=0.5pt}
\psset{nodesep=5pt}
\newpsstyle{a}{linecolor=red}
\newpsstyle{b}{linecolor=gray}
\newpsstyle{c}{linecolor=black}
\rput(16,0){\rnode[c]{1}{$1$}}
\rput(8,8){\rnode[c]{a}{$\tta$}}
\rput(16,8){\rnode[c]{b}{$\ttb$}}
\rput(32,8){\rnode[c]{c}{$\ttc$}}
\rput(0,16){\rnode[c]{ab}{$\tta\ttb$}}
\rput(8,16){\rnode[c]{bb}{$\ttb^2$}}
\rput(16,16){\rnode[c]{bc}{$\ttb\ttc$}}
\rput(24,16){\rnode[c]{aa}{$\tta^2$}}
\rput(32,16){\rnode[c]{cc}{$\ttc^2$}}
\rput(48,16){\rnode[c]{ca}{$\ttc\tta$}}
\rput(0,24){\rnode[c]{acc}{$\tta\ttc^2$}}
\rput(8,24){\rnode[c]{bba}{$\ttb^2\tta$}}
\rput(16,24){\rnode[c]{abb}{$\tta\ttb^2$}}
\rput(24,24){\rnode[c]{aaa}{$\tta^3$}}
\rput(32,24){\rnode[c]{baa}{$\ttb\tta^2$}}
\rput(40,24){\rnode[c]{aab}{$\tta^2\ttb$}}
\rput(48,24){\rnode[c]{caa}{$\ttc\tta^2$}}
\rput(0,32){\rnode[c]{bbaa}{$\ttb^2\tta^2$}}
\rput(16,32){\rnode[c]{aaaa}{$\tta^4$}}
\rput(24,32){\rnode[c]{bbbb}{$\ttb^4$}}
\rput(32,32){\rnode[c]{aacc}{$\tta^2\ttc^2$}}
\rput(40,32){\rnode[c]{cccc}{$\ttc^4$}}
\rput(48,32){\rnode[c]{ccbb}{$\ttc^2\ttb^2$}}
\rput(16,40){\rnode[c]{bbbbb}{$\ttb^5$}}
\rput(32,40){\rnode[c]{aaaaa}{$\tta^5$}}
\rput(40,40){\rnode[c]{ccccc}{$\ttc^5$}}
\rput(32,48){\rnode[c]{Delta}{$\Delta$}}
\psset{nodesep=0.4mm}
\CArrow(1,c)
\AArrow(c,ca)
\AArrow(a,aa)
\BArrow(aa,aab)
\BArrow(ab,abb)
\CArrow(abb,aacc)
\AArrow(1,a)
\BArrow(a,ab)
\BArrow(c,aa)
\CArrow(aa,abb)
\CArrow(ca,aab)
\AArrow(aab,aacc)

\AArrow(ab,acc)
\BArrow(acc,bbaa)
\BArrow(abb,aaaa)
\CArrow(aaaa,bbbbb)
\CArrow(aacc,aaaaa)
\AArrow(aaaaa,Delta)
\CArrow(acc,aaaa)
\AArrow(aaaa,aaaaa)
\AArrow(bbaa,bbbbb)
\BArrow(bbbbb,Delta)

\AArrow(ca,caa)
\BArrow(caa,ccbb)
\BArrow(aab,cccc)
\CArrow(cccc,ccccc)
\AArrow(caa,cccc)
\BArrow(cccc,aaaaa)
\BArrow(ccbb,ccccc)
\CArrow(ccccc,Delta)

\AArrow(b,cc)
\BArrow(cc,caa)
\BArrow(bb,aaa)
\CArrow(aaa,cccc)
\BArrow(b,bb)
\CArrow(bb,acc)
\CArrow(cc,aaa)
\AArrow(aaa,aaaa)

\CArrow(bba,bbbb)
\AArrow(bbbb,ccccc)
\AArrow(bb,bba)
\AArrow(aa,aaa)
\BArrow(aaa,bbbb)

\BArrow(1,b)
\CArrow(a,bb)
\CArrow(b,bc)
\CArrow(c,cc)
\CArrow(bc,bba)
\AArrow(cc,baa)
\AArrow(bba,bbaa)
\AArrow(baa,bbbb)
\BArrow(bc,baa)
\CArrow(baa,ccbb)
\BArrow(bbbb,bbbbb)
\end{picture}$$
\caption[]{\sf\smaller The Coxeter-like group associated with the RC-quasigroup of Example~\ref{X:Lattice}; the $27$-vertex cube shown above is the lattice of divisors of~$\tta^6$ in the associated monoid~$\MM$, the Hasse diagram of the weak order on the finite group~$\GGb$ with respect to the generators~$\tta, \ttb, \ttc$, and the Cayley graph of the germ derived from~$\GGb$ with respect to the previous generators. The complete Cayley graph of~$\GGb$ would be obtained by adding arrows that correspond to cases when the $\SS$-lengths do not add, for instance $\tta\ttb \cdot \ttc = 1$, as when one transforms a cube into a 3-torus by gluing opposite faces.}
\label{F:Cube}
\end{figure}
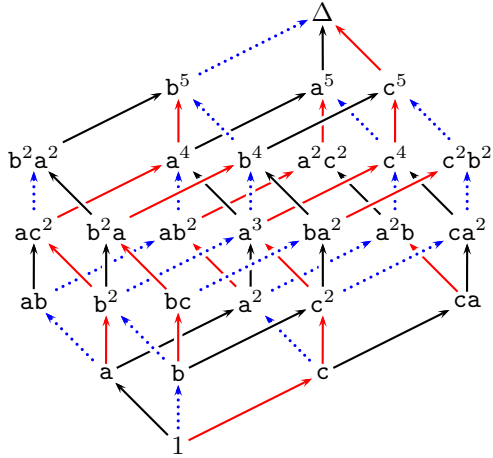

\begin{rema}
\ITEM1 General results by Gromov imply that every finitely generated group~$\GG$ whose Cayley graph is (quasi)-isometric to that of~$\ZZZZ^\nn$ must be virtually~$\ZZZZ^\nn$, that is, there exists an exact sequence $1 \to \ZZZZ^\nn \to \GG \to \GGb \to 1$ with $\GGb$ finite, see~\cite{BrGe}. By definition, an $I$-structure is an isometry as above, and the existence of a finite quotient~$\GGb$ as in Proposition~\ref{P:Cox} can be seen as a concrete instance of the above (abstract) result.

\ITEM2\label{R:IYB}
In~\cite{Eti}, one considers the quotient $G_{\!\SS}^0 = \GG{/}\Gamma$, where $\Gamma$ consists of the elements whose action on~$\SS$ is trivial. By Lemma~\ref{L:Frozen}, $\psi(\ss^\dd)$ is the identity permutation for every~$\ss$ in~$\SS$, meaning that every element~$\ss^{[\dd]}$ belongs to~$\Gamma$. Hence $G_{\!\SS}^0$ is a quotient of our current group~$\GGb$, a proper one in general: for the groups of Example~\ref{X:Cox}, in the class~$2$ example, the orders of~$\GG_\SS^0$ and~$\GGb$ are~$2$ vs.\ $4$ ($\tta\ttb$ lies in~$\Gamma$ but it is not trivial in~$\GGb$), in the class~3 example, the orders are $3$ vs.\ $27$.
\end{rema}

The question naturally arises of characterizing Coxeter-like groups associated with finite RC-quasigroups (hence, equivalently, with solutions of~YBE) as described in Proposition~\ref{P:Cox}. We show now that, exactly as structure groups of solutions of~YBE are those groups that admit an $I$-structure, their Coxeter-like quotients are those finite groups that admit the counterpart of an $I$-structure where some cyclic group~$\Zd$ replaces~$\ZZZZ$, that is, what was called an $I$-structure of shape~$(\Zd)^\SS$ in Definition~\ref{D:IStr}. 

\begin{prop}
\label{P:CharCox}
For every finite group~$\WW$, the following are equivalent:

\ITEM1 There exists a finite RC-quasigroup~$(\SS, \op)$ of class~Ê$\dd$ such that $\WW$ is the Coxeter-like group associated with~$(\SS, \op)$ and~$\dd$.

\ITEM2 The group~$\WW$ admits a right-$I$-structure of shape~$(\Zd)^\SS$.
\end{prop}

\begin{proof}
Assume that $(\SS, \op)$ is a finite RC-quasigroup and $\GG$ and $\MM$ are the associated group and monoid. Then $\MM$ admits an $I$-structure~$\nu$. The congruences of Lemma~\ref{L:Congruence}, $\equivz$ on~$\ZZZZ^\SS$ and~$\equiv$ on~$\GG$, are compatible and $\nu$ induces a well-defined map~$\nub$ from~$(\Zd)^\SS$ to~$\GG{/}{\equiv}$ that makes~\eqref{E:Projection} commutative. Then $\nub$ is bijective by construction, and projecting \eqref{E:IStrBis} for~$\nu$ gives its counterpart for~$\nub$. So $\nub$ is a right-$I$-structure of shape~$(\Zd)^\SS$ for~$\GG{/}{\equiv}$, and \ITEM1 implies~Ê\ITEM2.

Conversely, assume that $\WW$ admits a right-$I$-structure of shape~$(\Zd)^\SS$. As $(\Zd)^\SS$ satisfies~\eqref{E:Condition}, Lemmas~\ref{L:NuRC} and~\ref{E:NuRCBis} apply. Thus putting $\ss \op \tt = \psi(\ss)(\tt)$ yields an RC-quasigroup and \eqref{E:Nu} is valid. Let $\ss$ and~$\tt$ belong to~$\SS$. In~$(\Zd)^\SS$, we have $\ss^\dd = 1$, whence $\psi(\ss^\dd) = \psi(1) = \ID_\SS$. Applying~\eqref{E:Nu} (left), we deduce $\Pol_{\dd+1}(\ss \wdots \ss, \tt) = \psi(\ss^\dd)(\tt) = \tt$. Hence the RC-quasigroup $(\SS, \op)$ is of class~$\dd$. 

Now, let $\GG$ be the group associated with~$(\SS, \op)$, and let $\GGb$ be the associated finite quotient as provided by Proposition~\ref{P:Cox}. The group~$\GG$ is generated by~$\SS$ and it admits a presentation consisting of all relations $\ss(\ss \op \tt) = \tt(\tt \op \ss)$ with $\ss, \tt$ in~$\SS$. By assumption, the group~$\WW$ is generated by~$\SS$, and we observed that the relations $\ss(\ss \op \tt) = \tt(\tt \op \ss)$ with $\ss, \tt$ in~$\SS$ are satisfied in~$\WW$ since, by definition, they are equivalent to $\nu(\ss\tt) = \nu(\tt\ss)$ and $(\Zd)^\SS$ is Abelian. Hence there exists a surjective homomorphism~$\theta$ from~$\GG$ to~$\WW$ that is the identity on~$\SS$. 

$$\begin{picture}(50,20)(-3,-7)
\rput[bc](0,0){\rnode{00}{$(\Zd)^\SS$}}
\rput[bc](25,0){\rnode{10}{$\phantom{(}\GGb$\phantom{)}}}
\rput[c](0,12){\rnode{01}{$\ZZZZ^\SS$}}
\rput[c](25,12){\rnode{11}{$\phantom{(}\GG\phantom{)}$}}
\rput[c](50,0){\rnode{20}{$\phantom{(}\WW\phantom{)}$}}
\ncline[nodesep=2pt,linewidth=0.5pt]{->}{01}{11}\naput{$\widehat\nu$}
\ncline[nodesep=2pt,linewidth=0.5pt]{->}{01}{00}\tlput{$\piz$}
\ncline[nodesep=2pt,linewidth=0.5pt]{->}{11}{10}\trput{$\pi$}
\ncline[nodesep=2pt,linewidth=0.5pt]{->}{00}{10}
\ncline[nodesep=2pt,linewidth=0.5pt,style=exist]{->}{10}{20}\naput{$\overline\theta$}
\ncline[nodesep=2pt,linewidth=0.5pt,style=exist]{->}{11}{20}\naput{$\theta$}
\pscurve[linewidth=0.5pt]{->}(5,-2)(25,-5)(47,-1.5)
\put(24,-7.5){$\nu$}
\end{picture}$$

Next, $\GGb$ is the quotient of~$\GG$ obtained by adding the relations $\ss^{[\dd]} = 1$, that is, $\SPol_\dd(\ss \wdots \ss) = 1$ for~$\ss$ in~$\SS$. Now, in~$\WW$, we have $\SPol_\dd(\ss \wdots \ss) = \nu(\ss^\dd)$ by~\eqref{E:Nu}; but $\ss^\dd = 1$ holds in~$(\Zd)^\SS$, so, in~$\WW$, we have $\SPol_\dd(\ss \wdots \ss) = 1$ for every~$\ss$ in~$\SS$. Thus the surjective homomorphism~$\theta$ factorizes through~$\GGb$, yielding a surjective homomorphism~$\overline\theta$ from~$\GGb$ to~$\WW$. As both $\GGb$ and~$\WW$ have cardinality~$\dd^\nn$, the homomorphism~$\overline\theta$ must be an isomorphism. Hence $\WW$ is the Coxeter-like quotient of the group associated with~$(\SS, \op)$ and~$\dd$. So \ITEM2 implies~\ITEM1.
\end{proof}

Using Proposition~\ref{P:Wreath}, we deduce

\begin{coro}
\label{C:CoxWreath}
Every finite group that is the Coxeter-like group associated with an RC-quasigroup of size~$\nn$ and class~$\dd$ embeds into the wreath product $(\Zd) \wr \Sym_\nn$ so that the first component is a bijection.
\end{coro}

\begin{exam}
For the last group~$\GGb$ of Example~\ref{X:Cox}, owing to the fact that the permutations of~$\{1, 2, 3\}$ associated with $\tta, \ttb, \ttc$ all are the cycle $\ff : 1 \mapsto 2 \mapsto 3 \mapsto 1$, one obtains a description as the family of the 27 tuples $(\pp, \qq, \rr; \ff^{\pp+\qq+\rr})$ with $\pp, \qq, \rr$ in~$\ZZZZ{/}3\ZZZZ$, the product of triples     being twisted by the action of~$\ff^{\pp + \qq + \rr}$ on positions.
\end{exam}

Corollary~\ref{C:CoxWreath} implies that the Coxeter-like groups~$\GGb$ associated with finite RC-quasigroups are $IG$-monoids in the sense of~\cite{GoJe}. It follows that they inherit all properties of such monoids established there, in particular in terms of the derived algebras~$K[\GGb]$ and their prime ideals.

We conclude with linear representations of the groups~$\GG$ and~$\GGb$ associated with a finite RC-quasigroup. Here again, we use the $I$-structure to carry the results from the trivial case of a free Abelian group to the group of an arbitrary RC-quasigroup.

\begin{prop}
\label{P:LinRep}
Assume that $(\SS, \op)$ is an RC-quasi\-group of cardinal~$\nn$ and class~$\dd$. Let $\GG$ be the associated group. For~$\ss$ the $\ii$th element of~$\SS$, define
\begin{equation}
\label{E:LinRep}
\linrep(\ss) = \linrepz(\ss) \PP_{\psi(\ss)},
\end{equation}
where $\linrepz(\ss)$ is the diagonal $\nn \times \nn$-matrix with diagonal entries $(1 \wdots 1, \qq, 1 \wdots 1)$, $\qq$ at position~$\ii$ and $\PP_{\psi(\ss)}$ is the permutation matrix associated with~$\psi(\ss): \tt \mapsto \ss\op\tt$. Then $\linrep$ provides a faithful representation of~$\GG$ into $\GL(\nn, \QQQQ[\qq, \qq\inv])$; specializing at $\qq = \exp(2i\pi/\dd)$ gives a faithful representation of the associated group~$\GGb$.
\end{prop}

\begin{proof}
First, $\linrepz$ defines a faithful representation of~$\ZZZZ^\SS$ into $\GL(\nn, \QQQQ[\qq, \qq\inv])$ since $\linrepz(\prod\tts_\ii^{\ee_\ii})$ is the diagonal matrix with diagonal $(\qq^{\ee_1} \wdots \qq^{\ee_\nn})$, and specializing at $\qq = \exp(2i\pi/\dd)$ gives a faithful representation of~$(\Zd)^\SS$.

In order to carry the results to~$\GG$ and~$\GGb$, we show that \eqref{E:LinRep} extends into $\linrep(\nu(\aa)) = \linrepz(\aa) \PP_{\psi(\aa)}$ for every~$\aa$ in~$\ZZZZ^{\SS}$. As we are working with invertible matrices, it is enough to consider multiplication by one element of~$\SS$ (division automatically follows) and, therefore, the point for an induction is to go from~$\aa$ to~$\aa\ss$. We find
\begin{align*}
\linrep(\nu(\gga \ss)) 
&= \linrep(\nu(\gga) \psi(\gga)(\ss)) 
&\text{by \eqref{E:IStrBis}}\\
&= \linrep(\nu(\gga)) \, \linrep(\psi(\gga)(\ss))
&\text{by definition}\\
&= \linrepz(\gga) \, \PP_{\psi(\gga)} \, \linrepz(\psi(\gga)(\ss)) \, \PP_{\psi(\psi(\gga)(\ss))}
&\text{by induction hypothesis}\\
&= \linrepz(\gga) \, \linrepz(\ss) \, \PP_{\psi(\gga)} \, \PP_{\psi(\psi(\gga)(\ss))}
&\text{by conjugating by $\PP_{\psi(\gga)}$}\\
&= \linrepz(\gga \ss) \, \PP_{\psi(\psi(\gga)(\ss)) \comp \psi(\gga)}
&\text{by definition}\\
&= \linrepz(\gga \ss) \, \PP_{\psi(\gga\ss)}.
&\hspace{30mm}\text{by~\eqref{E:Product}}
\end{align*}
It is then clear that $\linrep$ is well-defined on the monoid associated with~$(\SS, \op)$, hence on its group of fractions, which is~$\GG$. For faithfulness, $\linrepz(\nu\inv(\gg))$ is the unique diagonal matrix obtained from~$\linrep(\gg)$ by right-multiplication by a permutation matrix, so $\linrep(\gg)$ determines~$\nu\inv(\gg)$, hence~$\gg$. 

Finally, specializing at a $\dd$th root of unity induces a well-defined faithful representation of the finite group~$\GGb$ since, by definition, $\gg$ and~$\gg'$ represent the same element of~$\GGb$ if and only if $\nu\inv(\gg)$ and $\nu\inv(\gg')$ are $\equivz$-equivalent, hence if and only if $\linrepz(\nu\inv(\gg))_{\qq = \exp(2i\pi/\dd)}$ and $\linrepz(\nu\inv(\gg'))_{\qq = \exp(2i\pi/\dd)}$ are equal.
\end{proof}

\begin{exam}
Coming back to the last group in Example~\ref{X:Cox}, the permutations $\psi(\tta)$, $\psi(\ttb)$, $\psi(\ttc)$ all are the $3$-cycle $(1, 2, 3)$, and we find the explicit representation
$$\linrep(\tta) = \begin{pmatrix}0&\qq&0\\0&0&1\\1&0&0\end{pmatrix}, \quad \linrep(\ttb) = \begin{pmatrix}0&1&0\\0&0&\qq\\1&0&0\end{pmatrix}, \quad \linrep(\ttc) = \begin{pmatrix}0&1&0\\0&0&1\\\qq&0&0\end{pmatrix}.$$
Specializing at $\qq = \exp(2i\pi/3)$ gives a faithful unitary representation of the associated $27$-element group~$\GGb$. Using the latter, it is easy to check for instance that $\GGb$ has exponent~$9$: $\tta, \ttb, \ttc$ have order~$9$, and all elements have order~$1$, $3$, or~$9$.
\end{exam}

\begin{coro}
\label{C:Isometries}
Every finite group that is the Coxeter-like group associated with an RC-quasigroup of size~$\nn$ can be realized as a group of isometries in an $\nn$-dimensional Hermitian space.
\end{coro}

Indeed, the matrices $\linrepz(\ss)_{\qq = \exp(2i\pi/\dd)}$ correspond to order~$\dd$ complex reflections, whereas permutation matrices are finite products of hyperplane symmetries. 

We shall not go farther in the description of the Coxeter-like groups~$\GGb$. The analogy with Coxeter groups suggests to further investigate their geometric properties: according to Proposition~\ref{P:CharCox} and Figure~\ref{F:Cube}, the Cayley graph of the Coxeter-like group associated with an RC-quasigroup of cardinal~$\nn$ and class~$\dd$ is drawn on an $\nn$-torus and corresponds to a tiling of the torus (hence to a periodic tiling of~$\RRRR^\nn$) by $\dd^\nn$~copies of a single cubical pattern. On the other hand, as $\GGb$ characterizes the RC-quasigroup it comes from, classifying all finite groups (or all periodic tilings of~$\RRRR^\nn$) that occur in this way is \textit{a priori} not easier than classifying all solutions of~YBE, hence presumably (very) difficult.

We conclude with a speculative idea. So far, Propositions~\ref{P:Charac} and \ref{P:RCItype} provide the only known characterization of a relatively large family of Garside groups: it identifies Garside groups that admit a presentation of a certain form with those that admit an $I$-structure, hence resemble a free Abelian group. One might replace free Abelian groups with other groups~$\Gamma$ and consider as in~\cite{GoJe} those groups~$\GG$ that admit a ``$I$-structure of shape~$\Gamma$'' in the sense that their Cayley graph is that of~$\Gamma$ up to relabeling the edges. Should this approach make sense here, a natural problem would be to characterize those Garside groups that admit an $I$-structure of shape~$\Gamma$ for various reference (Garside) groups~$\Gamma$, maybe a first step toward a global classification of Garside groups which, so far, remains out of reach.

\end{document}